\documentclass[a4paper,11pt]{article}
\usepackage[latin1]{inputenc}
\usepackage[english]{babel}
\usepackage{amsmath}
\usepackage{amsfonts}
\usepackage{amssymb}
\usepackage{epsfig}
\usepackage{amsopn}
\usepackage{amsthm}
\usepackage{color}
\usepackage{graphicx}
\usepackage{subfigure}
\usepackage{enumerate}
\newtheorem{theorem}{Theorem}[section]

\newtheorem{lemma}[theorem]{Lemma}
\newtheorem{proposition}[theorem]{Proposition}
\newtheorem{definition}[theorem]{Definition}

\newtheorem*{theorem*}{Theorem}
\newtheorem*{lemma*}{Lemma}
\newtheorem*{remark*}{Remark}
\newtheorem*{definition*}{Definition}
\newtheorem*{proposition*}{Proposition}
\newtheorem*{corollary*}{Corollary}
\numberwithin{equation}{section}
\newcommand{\real}{\mathbb{R}}

\def\qed{\,\unskip\kern 6pt \penalty 500
\raise -2pt\hbox{\vrule \vbox to8pt{\hrule width 6pt
\vfill\hrule}\vrule}\par}
\definecolor{darkblue}{rgb}{0.05, .05, .65}
\definecolor{darkgreen}{rgb}{0.1, .65, .1}
\definecolor{darkred}{rgb}{0.8,0,0}
\newcommand{\beqn}{\begin{equation}}
\newcommand{\eeqn}{\end{equation}}
\newcommand{\bear}{\begin{eqnarray}}
\newcommand{\eear}{\end{eqnarray}}
\newcommand{\bean}{\begin{eqnarray*}}
\newcommand{\eean}{\end{eqnarray*}}
\begin{document}

\title{\huge \bf Separate variable blow-up patterns for a reaction-diffusion equation with critical weighted reaction}

\author{
\Large Razvan Gabriel Iagar\,\footnote{Departamento de Matem\'{a}tica
Aplicada, Ciencia e Ingenieria de los Materiales y Tecnologia
Electr\'onica, Universidad Rey Juan Carlos, M\'{o}stoles,
28933, Madrid, Spain, \textit{e-mail:} razvan.iagar@urjc.es},
\\[4pt] \Large Ariel S\'{a}nchez,\footnote{Departamento de Matem\'{a}tica
Aplicada, Ciencia e Ingenieria de los Materiales y Tecnologia
Electr\'onica, Universidad Rey Juan Carlos, M\'{o}stoles,
28933, Madrid, Spain, \textit{e-mail:} ariel.sanchez@urjc.es}\\
[4pt] }
\date{}
\maketitle

\begin{abstract}
We study the separate variable blow-up patterns associated to the following second order reaction-diffusion equation:
$$
\partial_tu=\Delta u^m + |x|^{\sigma}u^m,
$$
posed for $x\in\real^N$, $t\geq0$, where $m>1$, dimension $N\geq2$ and $\sigma>0$. A new and explicit critical exponent
$$
\sigma_c=\frac{2(m-1)(N-1)}{3m+1}
$$
is introduced and a classification of the blow-up profiles is given. The most interesting contribution of the paper is showing that existence and behavior of the blow-up patterns is split into different regimes by the critical exponent $\sigma_c$ and also depends strongly on whether the dimension $N\geq4$ or $N\in\{2,3\}$. These results extend previous works of the authors in dimension $N=1$.
\end{abstract}

\


\smallskip


\section{Introduction}

The aim of this paper is to study the blow-up patterns (in the form of separate variable solutions) to the following reaction-diffusion equation
\begin{equation}\label{eq1}
\partial_tu=\Delta u^m+|x|^{\sigma}u^m,
\end{equation}
posed for $(x,t)\in\real^N\times(0,T)$ for some $T>0$, in dimension $N\geq2$ and with $m>1$, $\sigma>0$. It is well known (see for example the book \cite{S4} for the homogeneous case $\sigma=0$ or the recent paper by the authors \cite{IS20b} in one dimension) that finite time blow-up is expected to occur for solutions to \eqref{eq1}, which means that there exists $T\in(0,\infty)$ such that $u(t)\in L^{\infty}(\real)$ for any $t\in(0,T)$ but $u(T)\not\in L^{\infty}(\real)$. The time $T\in(0,\infty)$ satisfying this property is called the blow-up time of the solution $u$. Here and throughout the paper we employ the short notation $u(t)$ for the mapping $x\mapsto u(x,t)$ for a fixed time $t\geq0$. The blow-up profiles to Eq. \eqref{eq1} in dimension $N=1$ have been studied and classified in our previous work \cite{IS20b} but, as we show in the present paper, the ranges of existence or non-existence and the behavior of the separate variable profiles \emph{is strikingly different in higher dimensions} and some new critical parameters of the problem (both with respect to the exponent $\sigma$ and to the dimension $N$) will be introduced.

The interest for the weighted reaction-diffusion equation with unbounded, power-like weight, having the general form
\begin{equation}\label{eq.nohom.gen}
u_t=\Delta u^m+|x|^{\sigma}u^p,
\end{equation}
comes from the study of the influence that the weight has on the qualitative properties of the solutions, the dynamics of the equation and the blow-up behavior of it. The first question, addressed by mathematicians such as Bandle, Levine, Baras, Kersner et. al. was to establish criteria, depending on $m$, $p$, $\sigma$, the dimension $N$ and the initial condition $u(x,0)$, on whether the solutions to Eq. \eqref{eq.nohom.gen} blow-up in finite time or not. We quote a series of works devoted to this problem and emphasizing on the "life-span" of solutions (that is, understanding how the blow up time of a one-parameter family of solutions changes with respect to the parameter) for the semilinear case $m=1$ \cite{BL89, BK87, Pi97, Pi98}, where more general weights $a(x)$ instead of the pure power $|x|^{\sigma}$ are considered. The same problem, but for the quasilinear case $m>1$, has been addressed by Suzuki in \cite{Su02}. Suzuki established in his paper the Fujita-type exponent $p_{F}^{m,\sigma}:=m+(\sigma+2)/N$, that is, the minimal exponent $p>1$ such that, for any $1<p<p_F^{m,\sigma}$, all the solutions blow up in finite time, the name being given by similarity to the seminal paper by Fujita \cite{Fu66} where such an exponent has been introduced for the homogeneous problem. The same author also obtained sufficient conditions on the tails of the initial data $u(x,0)$ as $|x|\to\infty$ for the finite time blow-up to occur when $p>p_{F}^{m,\sigma}$. A blow-up rate of the solutions to Eq. \eqref{eq.nohom.gen} has been derived through functional estimates in \cite{AT05} in the range of exponents $m<p<m+N/2$ and $0<\sigma\leq N(p-m)/m$.

A different but very interesting question is whether the zeros of the weight (where formally there is no reaction) can be blow-up points for the solutions. Some answers to this question were given in the semilinear case $m=1$ in the series of papers \cite{GLS, GS11, GLS13, GS18} where it is shown that in general, most solutions cannot blow up at $x=0$, but the origin can be still a blow-up point in some very specific cases. The latter paper \cite{GS18} extends these results to more general weights $a(x)$ instead of pure powers $|x|^{\sigma}$ giving conditions under which the zeros of $a(x)$ cannot be blow-up points for solutions. The study of self-similar blow-up profiles to Eq. \eqref{eq.nohom.gen} for $m>1$ and weights $|x|^{\sigma}$ gave an answer to this question for the quasilinear range $m>1$ in our recent works \cite{IS19, IS21}. In these papers, we considered Eq. \eqref{eq.nohom.gen} in dimension $N=1$ and with $1\leq p<m$ and we proved that for $\sigma>0$ sufficiently close to zero (more exactly, in some interval $\sigma\in[0,\sigma^*]$) there exist blow-up profiles with simultaneous and complete blow-up, thus including $x=0$ as a blow-up point. This is a significant difference with respect to the semilinear case $m=1$ and another interesting outcome of the analysis of blow-up profiles in self-similar form.

The specific case of Eq. \eqref{eq1} proved to be very interesting to study, playing the role of a critical one between two ranges of exponents with highly different dynamics. Indeed, as indicated in \cite[Chapter 4]{S4} for the homogeneous case $\sigma=0$ and in \cite{IS19, IS21} for $\sigma>0$, the regime $1\leq p<m$ is characterized by compactly supported solutions: all the interesting profiles and general solutions remain compactly supported up to the blow-up time, and \cite{IS19} for $p=1$ and \cite{IS21} for $1<p<m$ give a classification of the blow-up profiles in backward self-similar form for Eq. \eqref{eq.nohom.gen} in dimension $N=1$. Meanwhile, the regime $p>m$ is characterized by solutions with tails as $|x|\to\infty$, as it comes out from \cite[Chapter 4]{S4} for $\sigma=0$ and from \cite{Su02} for $\sigma>0$. The dynamics of the equation should be self-similar also in this range $p>m$, as it follows from the study performed in dimension $N=1$ in \cite[Chapter 4]{S4}, but an understanding of it is still missing in higher dimensions or for any $\sigma>0$. We mention here that even for $\sigma=0$, with dimension $N$ sufficiently large and $p>m$ also larger, \emph{existence or non-existence of blow-up profiles is still an open problem}, the critical exponents in terms of $p$ and $N$ limiting the well-studied and the open ranges can be consulted in \cite{S4, GV97}. This is why, we also believe that the present work extending our study in dimension $N=1$ \cite{IS20b} to any space dimension $N\geq2$, gives us experience and a further level of understanding in a "neighbor problem" to the above-mentioned one, which will be considered in a forthcoming paper making use of the techniques and ideas we learnt and developed for the current paper. We also mention that the blow-up set of solutions to Eq. \eqref{eq1} with $\sigma=0$ has been analyzed in \cite{CDE98}.

To end this presentation of the precedents, we also mention the works \cite{FdPV06} in dimension $N=1$ and \cite{KWZ11, Liang12, FdP18} for $N\geq2$, in which a similar equation to Eq. \eqref{eq1} (that is, with $p=m$) is considered, but replacing $|x|^{\sigma}$ by a localized, compactly supported and bounded weight $a(x)$. It is there shown that the solutions may sometimes be global and present grow up instead of blow up, depending on the support of $a(x)$. For the one-dimensional case, the work \cite{FdPV06} goes into a deeper study by establishing blow up rates, sets and profiles.

We describe below the most significant contributions of the paper.

\medskip

\noindent \textbf{Main results.} We focus in this paper on the analysis and classification of solutions in separate variable form to Eq. \eqref{eq1}, having the precise form
\begin{equation}\label{Sepsol}
u(x,t)=(T-t)^{-\alpha}f(|x|), \qquad \alpha=\frac{1}{m-1},
\end{equation}
where the profiles $f$ solve the following differential equation
\begin{equation}\label{ODE}
(f^m)''(\xi)+\frac{N-1}{\xi}(f^m)'(\xi)-\frac{1}{m-1}f(\xi)+\xi^{\sigma}f^m(\xi)=0, \qquad \xi=|x|.
\end{equation}
As seen also in \cite{IS20b} in dimension $N=1$, this is the particular form of the self-similar blow-up patterns for the critical case $p=m$, where the supports of the solutions are localized and fixed. Thus, our goal is to perform a thorough study of the differential equation \eqref{ODE}, obtain a classification of the blow-up profiles with respect to their behavior and thus extract valuable knowledge on the dynamics of Eq. \eqref{eq1}, since it is by now well-established that the self-similar profiles for a reaction-diffusion equation are fundamental in understanding its general dynamics for many reasons (profiles for the geometric form of the solutions when approaching their blow-up time, examples of true solutions with specific behavior, optimizing estimates and inequalities on general solutions). As we shall see, the study in higher dimensions $N\geq2$ \emph{strikingly departs} from the one in dimension $N=1$, giving us also a better understanding of why $N=1$ and $\sigma=0$ are highly critical dimension and exponent. The explanation will become even more obvious with a change of variable that will be indicated later, but for now let us introduce a very important exponent for the whole study
\begin{equation}\label{crit.exp}
\sigma_c=\frac{2(N-1)(m-1)}{3m+1}.
\end{equation}
We notice that for $N=1$ we get $\sigma_c=0$, and this together with the subsequent analysis will show how critical is the homogeneous case in one dimension.

We will now state our main results. To this end, we define rigorously below what we understand by a "good solution".
\begin{definition}\label{def1}
We say that $f$ solution to \eqref{ODE} is a \textbf{good profile} if it fulfills one of the following three properties related to its initial behavior:

\indent (P1) $f(0)=a>0$, $f'(0)=0$.

\indent (P2) $f(0)=0$, $(f^m)'(0)=0$.

\indent (P3) There exists $\xi_0\in(0,\infty)$ such that $f(\xi_0)=0$, $(f^m)'(\xi_0)=0$ and $f(\xi)>0$ in a right-neighborhood of $\xi_0$.

A good profile $f$ is called a \textbf{good profile with interface} at some point $\eta\in(0,\infty)$ if $f(\eta)=0$, $(f^m)'(\eta)=0$, and $f(\xi)>0$ in a left-neighborhood of $\eta$.
\end{definition}
We are interested mainly in studying the good profiles with interface solutions to \eqref{ODE}. Good profiles presenting a tail as $\xi\to\infty$ instead of an interface will also be considered as of secondary interest if they exist. With respect to the good profiles with interface, their existence or non-existence and their local behavior strongly depend on the critical exponent $\sigma_c$ introduced in \eqref{crit.exp} and on whether the dimension $N>3$ or $N\leq3$. We begin with the statements of our main theorems for $N\geq4$.
\begin{theorem}[Existence of blow-up profiles in dimension $N\geq4$]\label{th.exist.N4}
Let $N\geq4$. Then good profiles with interface \emph{exist} for any $\sigma\in(0,\sigma_c)$ and \emph{do not exist} at least for any $\sigma\in[\sigma_c,2(N-3)]$.
\end{theorem}
Thus, the influence of the critical exponent in \eqref{crit.exp} is very sharp: in higher space dimensions it limits the regimes of existence and non-existence of good profiles with interface. This has a significant consequence for the equation: while for $\sigma<\sigma_c$ it is expected that any solution takes a separate variable pattern among the existing ones when approaching its blow-up time, for $\sigma\geq\sigma_c$ the problem of the dynamics near blow-up remains completely open and maybe asymptotic simplifications or blow-up only at the infinity of the space and not in separate variable form have to be considered. Since our goal is not only to establish the existence of the profiles, but to understand better their behavior too, our next theorem gives a (partial) classification of them in the range where they exist.
\begin{theorem}[Classification for $N\geq4$]\label{th.class.N4}
Let again $N\geq4$. Then there exist $\sigma_0$ and $\sigma^0\in(0,\sigma_c)$ such that
\begin{itemize}
\item For any $\sigma\in(0,\sigma_0)$, the good profiles with interface to Eq. \eqref{eq1} present a behavior at $\xi=0$ given by the assumption (P1) in Definition \ref{def1}.
\item For any $\sigma\in(\sigma^0,\sigma_c)$ the good profiles with interface to Eq. \eqref{eq1} present an initial behavior corresponding to the assumption (P3) in Definition \ref{def1}
\item For $\sigma=\sigma_0$, $\sigma=\sigma^0$ the good profiles with interface to Eq. \eqref{eq1} present a behavior at $\xi=0$ given by the assumption (P2) in Definition \ref{def1}. In this case, the local behavior of the profiles can be made more precise:
\begin{equation}\label{beh.P2}
f(\xi)\sim\left[\frac{m-1}{2m(mN-N+2)}\right]^{1/(m-1)}\xi^{2/(m-1)}, \qquad {\rm as} \ \xi\to0.
\end{equation}
\end{itemize}
\end{theorem}
It is a strong expectation that the set where the third assumption in Theorem \ref{th.class.N4} holds true has to be a singleton, that is, $\sigma_0=\sigma^0$. Proving this uniqueness requires to establish a monotonicity in the evolution of the whole dynamical system with respect to the parameter $\sigma$ and will be left open in the current work, as a conjecture that we believe it holds true. Notice also that the homogeneous case $\sigma=0$ studied in \cite[Chapter 4]{S4} is a particular case of the first range of Theorem \ref{th.class.N4}.

Returning now to lower space dimensions $N=2$, $N=3$, we shall see that the classification of the blow-up profiles strongly departs from the corresponding results in dimension $N\geq4$ and it is more similar to the one established in \cite{IS20b} for $N=1$, although with some differences. We make it more precise below.
\begin{theorem}[Existence and classification for $N=2$ and $N=3$]\label{th.exist.N23}
Let $N=2$ or $N=3$. Then, for any $\sigma\in[0,\sigma_c]$ there exist good blow-up profiles with interface. \textbf{All} these profiles present a behavior at $\xi=0$ given by the assumption (P1) in Definition \ref{def1}, that is, $f(0)=A>0$ with $f'(0)=0$.
\end{theorem}
In fact, the range of existence of good blow-up profiles with interface can be further extended above $\sigma_c$, a thing that makes a strong difference with respect to higher space dimensions, as noticed in Theorem \ref{th.exist.N4}. In practice, numerical experiments suggest that the upper limit of this range is a very close number to $\sigma_c$, thus the interval of existence of good blow-up profiles can be extended shortly after $\sigma_c$, but for $\sigma$ larger they are expected to cease to exist. We recall here that a rather similar classification has been established in \cite{IS20b} in dimension $N=1$, starting from the critical exponent $\sigma_c=0$, and some precise estimates in terms of $m$ on the exponent $\sigma>\sigma_c$ limiting the existence and non-existence ranges have been established there. Finally, for $\sigma$ sufficiently large the profiles indeed cease to exist, and this qualitative aspect does not depend on the space dimension.
\begin{theorem}[Non-existence for $\sigma$ large]\label{th.nonexist}
In the previous notation, there exists $\sigma_1\geq\sigma_c$ sufficiently large (depending on $m$ and $N$) such that for any $\sigma>\sigma_1$ there are no good blow-up profiles.
\end{theorem}

\noindent \textbf{Comments on an open problem.} We stress here that both the problems with variable and unbounded coefficients (from the PDE theory perspective) and involving a study of a three-dimensional dynamical system (from the point of view of the dynamical systems approach, that is our method for the proofs) are very difficult ones. We \emph{conjecture} here that in fact for $N\geq4$, there is no good blow-up profile for any $\sigma\geq\sigma_c$, thus in this range $\sigma_1=\sigma_c$ in Theorem \ref{th.exist.N4} (while this is clearly not true for $N=2$ and $N=3$, where the existence range of $\sigma$ can be extended slightly above $\sigma_c$, as previously discussed). However, completing the intermediate range $\sigma>2(N-3)$ for $N\geq4$ would require to establish some aspects both of \emph{monotonicity with respect to $\sigma$} of the evolution of the trajectories in the subsequent phase space, and of the theory of two-dimensional unstable manifolds in the space that seem by now very difficult or unavailable, up to our knowledge. We will address this question in forthcoming works.

\medskip

\noindent \textbf{A transformation.} We give here an heuristic transformation leading to a better understanding of why dimensions $N\leq3$ and $N\geq4$ bring different results. Assume that $N$ is just a parameter in the differential equation \eqref{ODE}, allowing thus formally for real dimensions (a convention that is rather common and useful when dealing with radially symmetric or self-similar solutions). Let then
\begin{equation}\label{critN}
N^*=\frac{4m+2}{m+1}\in(3,4), \qquad \sigma_c(N^*)=\frac{2(m-1)}{m+1}
\end{equation}
and introduce the transformation
\begin{equation}\label{transf}
\xi=\left(\frac{2m}{m+1}\eta\right)^{(m+1)/2m}, \qquad f(\xi)=F(\eta)\left(\frac{2m}{m+1}\eta\right)^{-1/m}.
\end{equation}
It is easy to check that, if $f(\xi)$ is a good profile with interface to Eq. \eqref{ODE} in dimension $N^*$ and with $\sigma=\sigma_c(N^*)$, then $F(\eta)$ is a good profile with interface for the homogeneous equation Eq. \eqref{eq1} with $\sigma=0$ and dimension $N=1$. This interesting (at a formal level) self-map of Eq. \eqref{eq1} suggests that dimensions $N\in[1,N^*]$ and dimensions $N>N^*$ form different regimes with respect to the dynamics of Eq. \eqref{eq1}. Since $N^*\in(3,4)$, we notice these differences when passing from $N=3$ to $N=4$ in physical (natural) dimensions. The critical value $N^*$ will be present in various calculations throughout the paper. Let us finally state here that most of our results keep holding true if we allow the convention to consider $N$ only as a \emph{real parameter in \eqref{ODE}} instead of a physical dimension, mentioning that the results we prove for dimensions $N=2$ and $N=3$ hold true with this convention for any $N\in(1,N^*)$, while the results we prove for physical dimensions $N\geq4$ hold true for $N>N^*$.

\medskip

\noindent \textbf{Structure of the paper}. The main technique used in the proofs is converting Eq. \eqref{ODE} into an autonomous quadratic dynamical system of three first order equations and then study the associated phase space employing results and techniques typical for dynamical systems. Sometimes, direct estimates obtained from the equation will be needed too. The first two technical Sections \ref{sec.local} and \ref{sec.infty} are devoted to the local analysis of the critical points in the phase space and at the infinity of it. We stress here that the presence of \emph{periodic orbits} is a significant difficulty in the analysis. Once performed the (quite involved) local analysis of the critical points, we pass to the proof of our main results, by driving the orbits in the phase space. After devoting a short preparatory section to the analysis of an invariant plane, the existence part in Theorem \ref{th.exist.N4} is proved in Section \ref{sec.exist}. The non-existence for $\sigma\geq\sigma_c$ (at least in some interval) is postponed to the end of Section \ref{sec.class}, which is mostly devoted to the proof of the classification of the blow-up patterns for $\sigma\in(0,\sigma_c)$ and $N\geq4$, but some technical constructions used for this classification are essential also in the proof of the non-existence statement. We go back to the special case of dimensions $N=2$ and $N=3$ and prove Theorem \ref{th.exist.N23} in Section \ref{sec.N23}. The paper is finally closed by the proof of the non-existence Theorem \ref{th.nonexist} for $\sigma$ sufficiently large, which is done in Section \ref{sec.nonexist} using a geometrical construction which does not depend on the space dimension.

\section{The phase space. Local analysis}\label{sec.local}

We introduce the following change of variables in order to transform the non-autonomous equation \eqref{ODE} into an autonomous dynamical system
\begin{equation}\label{PSchange}
\begin{split}
&X(\eta)=\sqrt{m(m-1)}\xi^{-1}f^{(m-1)/2}(\xi), \ \ Y(\eta)=\frac{2\sqrt{m(m-1)}}{m-1}(f^{(m-1)/2})'(\xi), \\ &Z(\eta)=(m-1)\xi^{\sigma}f^{m-1}(\xi),
\end{split}
\end{equation}
together with the change of independent variable given by
\begin{equation}\label{change2}
\frac{d\eta}{d\xi}=\frac{1}{\sqrt{m(m-1)}}f^{-(m-1)/2}(\xi),
\end{equation}
to transform Eq. \eqref{ODE} into the following quadratic autonomous system of differential equations
\begin{equation}\label{PSsyst}
\left\{\begin{array}{ll}\dot{X}=\frac{m-1}{2}XY-X^2,\\
\dot{Y}=-\frac{m+1}{2}Y^2+1-Z-(N-1)XY,\\
\dot{Z}=Z[(m-1)Y+\sigma X],\end{array}\right.
\end{equation}
where the derivative is taken with respect to the new variable $\eta$. Let us notice here that the change of variable is the same as the one used in dimension $N=1$ \cite{IS20b}, and the system \eqref{PSsyst} is also apparently very similar, the only new term being $(N-1)XY$ in the second equation. However, as we shall see from the subsequent analysis, this single term introduces very significant changes and difficulties in the study of the phase space. The critical points in the finite part of the phase plane are the following four
\begin{equation*}
P_0=(0,h_0,0), \ P_1=(0,-h_0,0), \ P_2=(X(P_2),Y(P_2),0) \ {\rm and} \ P_3=(0,0,1),
\end{equation*}
where $h_0=\sqrt{2/(m+1)}$ and
\begin{equation}\label{xp2}
X(P_2)=\frac{m-1}{\sqrt{2(mN-N+2)}}, \qquad Y(P_2)=\sqrt{\frac{2}{mN-N+2}}.
\end{equation}
We also notice that the planes $\{X=0\}$ and $\{Z=0\}$ are invariant for the system \eqref{PSsyst}, thus we restrict the analysis to the region $\{X\geq0, Z\geq0\}$, only the component $Y$ being allowed to change sign, which is coherent with the definition of $X$, $Y$, $Z$. The section is divided into subsections corresponding to the analysis of the points and of the periodic orbits.

\subsection{Analysis of the hyperbolic critical points}\label{subsec.hyp}

The hyperbolic critical points of the system are $P_0$, $P_1$ and $P_2$ and this subsection is dedicated to their local analysis, which is rather analogous to the one in dimension $N=1$ done in \cite[Section 2]{IS20b}. The local analysis near the point $P_3$ is postponed to the next subsection.
\begin{lemma}[Analysis of the points $P_0$ and $P_1$]\label{lem.P0P1}
The system \eqref{PSsyst} in a neighborhood of the critical point $P_0$ has a two-dimensional unstable manifold and a one-dimensional stable manifold. The orbits going out of $P_0$ on the unstable manifold contain profiles such that
\begin{equation}\label{behP0}
f(\xi)\sim\left(\frac{(m-1)h_0}{2\sqrt{m(m-1)}}\xi-K\right)_+^{2/(m-1)}, \quad K>0, \ \ {\rm as} \ \xi\to\xi_0=\frac{2K\sqrt{m(m-1)}}{(m-1)h_0},
\end{equation}
that is, profiles satisfying the assumption (P3) in Definition \ref{def1}. The system in a neighborhood of the critical point $P_1$ has a one-dimensional unstable manifold and a two-dimensional stable manifold. The orbits entering $P_1$ on the stable manifold contain profiles such that
\begin{equation}\label{behP1}
f(\xi)\sim\left(K-\frac{(m-1)h_0}{2\sqrt{m(m-1)}}\xi\right)_+^{2/(m-1)}, \quad K>0, \ \ {\rm as} \ \xi\to\xi_0=\frac{2K\sqrt{m(m-1)}}{(m-1)h_0},
\end{equation}
that is, profiles with interface at a positive point $\xi=\xi_0>0$ as described in Definition \ref{def1}.
\end{lemma}
\begin{proof}
The linearization of the system \eqref{PSsyst} in a neighborhood of the critical points $P_0$, respectively $P_1$, has the matrix
$$
M=\left(
         \begin{array}{ccc}
           \pm(m-1)h_0/2 & 0 & 0 \\
           \mp(N-1)h_0 & \mp(m+1)h_0 & -1 \\
           0 & 0 & \pm(m-1)h_0 \\
         \end{array}
       \right),$$
with eigenvalues $\lambda_1=\pm(m-1)h_0/2$, $\lambda_2=\mp(m+1)h_0$, $\lambda_3=\pm(m-1)h_0$, the plus sign corresponding to $P_0$ and the minus sign to $P_1$. Thus, $P_0$ has a two-dimensional unstable manifold, while $P_1$ has a two-dimensional stable manifold. It is rather obvious that the orbits entering $P_0$ on the one-dimensional stable manifold, or going out of $P_1$ on the one-dimensional unstable manifold (both corresponding to the eigenvalue $\lambda_2$) are contained in the $Y$ axis. The orbits going out of $P_0$ (respectively entering $P_1$) on the two-dimensional unstable manifold (respectively the two-dimensional stable manifold) contain profiles such that $X\to0$, $Y\to\pm h_0$ and $Z\to0$. We then infer from \eqref{PSchange} that on the one hand we have the local behavior
\begin{equation}\label{interm1}
(f^{(m-1)/2})'(\xi)\sim\pm\frac{(m-1)h_0}{2\sqrt{m(m-1)}},
\end{equation}
and that on the other hand \eqref{interm1} holds true as $\xi\to\xi_0\in(0,\infty)$. The latter is proved by discarding the possibilities that either $\xi\to0$ or $\xi\to\infty$ in \eqref{interm1}, which is established by translating the fact that $X(\xi)\to0$ or $Z(\xi)\to0$ in terms of profiles. The details are identical to the ones given in the proof of \cite[Lemma 2.1]{IS20b} and will be omitted here. We next find the behavior \eqref{behP0} by integration in \eqref{interm1} when working with the plus sign and the local behavior \eqref{behP1} by a similar integration when working with the minus sign in \eqref{interm1}.
\end{proof}
The local analysis of the point $P_2$ follows below.
\begin{lemma}[Analysis of the point $P_2$]\label{lem.P2}
The system in a neighborhood of the critical point $P_2$ has a two-dimensional stable manifold and a one-dimensional unstable manifold. The stable manifold is contained in the invariant plane $\{Z=0\}$. There exists a unique orbit going out of $P_2$, containing profiles with local behavior given by \eqref{beh.P2}.
\end{lemma}
\begin{proof}
The linearization of the system \eqref{PSsyst} near the critical point $P_2$ has the matrix
$$
M(P_2)=\left(
         \begin{array}{ccc}
           -\frac{m-1}{\sqrt{2(mN-N+2)}} & \frac{(m-1)^2}{2\sqrt{2(mN-N+2)}} & 0 \\
           -\frac{2(N-1)}{\sqrt{2(mN-N+2)}} & -\frac{2(m+1)+(N-1)(m-1)}{\sqrt{2(mN-N+2)}} & -1 \\
           0 & 0 & \frac{(m-1)(\sigma+2)}{\sqrt{2(mN-N+2)}} \\
         \end{array}
       \right),
$$
with eigenvalues $\lambda_1$, $\lambda_2$ and $\lambda_3$ satisfying
$$
\lambda_1+\lambda_2=-\frac{(m-1)N+2(m+1)}{\sqrt{2(mN-N+2)}}<0, \ \lambda_1\lambda_2=\frac{(m+1)(m-1)+(N-1)(m-1)^2}{mN-N+2}>0
$$
whence $\lambda_1<0$ and $\lambda_2<0$, and
\begin{equation}\label{lambda3}
\lambda_3=\frac{(m-1)(\sigma+2)}{\sqrt{2(mN-N+2)}}>0.
\end{equation}
It is easy to check, using the invariance of the plane $\{Z=0\}$ (similarly as in \cite[Lemma 2.3]{IS19}) that the two-dimensional stable manifold is contained in the plane $\{Z=0\}$ and there exists only one orbit going out of $P_2$ into the region $\{Z>0\}$ of the phase space, which goes out tangent to the direction of the eigenvector $e_3$ corresponding to the eigenvalue $\lambda_3$, that is
\begin{equation}\label{interm2}
e_3(\sigma)=\left(-\frac{(m-1)\sqrt{2(mN-N+2)}}{2l(\sigma)},\frac{(\sigma+3)\sqrt{2(mN-N+2)}}{l(\sigma)},1\right)
\end{equation}
with
$$
l(\sigma):=(m-1)\sigma^2+(mN+6m-N-2)\sigma+4(mN+2m-N+2).
$$
Despite its very tedious form, the precise dependence on $\sigma$ of the components of $e_3$ will be very useful in the global analysis of this orbit. Since on this unique orbit we have
$$
X(\xi)\to\frac{m-1}{\sqrt{2(mN-N+2)}}, \ Y(\xi)\to\sqrt{\frac{2}{mN-N+2}}, \ Z(\xi)\to0,
$$
we readily deduce from the formula for $X(\xi)$ in \eqref{PSchange} that the profiles contained in this unique orbit behave locally as in \eqref{beh.P2}. The fact that the behavior in \eqref{beh.P2} is taken as $\xi\to0$ follows from discarding the possibilities $\xi\to\xi_0\in(0,\infty)$ or $\xi\to\infty$. Assuming for contradiction that the former takes place and recalling that $Z(\xi)\to0$, we get that $f(\xi)\to0$ as $\xi\to\xi_0$, contradicting that $X(\xi)$ tends to a positive constant. If we assume for contradiction that the convergence in \eqref{beh.P2} takes place as $\xi\to\infty$, again $Z(\xi)\to0$ implies that $f(\xi)\to0$ as $\xi\to\infty$, then $X(\xi)\to0$ and a contradiction. We thus remain with the behavior \eqref{beh.P2} as $\xi\to0$, as stated.
\end{proof}

\subsection{Analysis of the non-hyperbolic point $P_3$}\label{subsec.P3}

This subsection is devoted to the analysis of the critical point $P_3$, which is a non-hyperbolic one, the linearization of the system \eqref{PSsyst} in a neighborhood of this point having three eigenvalues with zero real part: $\lambda_1=0$, $\lambda_2=i\sqrt{m-1}$, $\lambda_3=-i\sqrt{m-1}$. Thus, in order to analyze the system in a neighborhood of $P_3$ we have to use more involved techniques specific to the bifurcation theory, more precisely by deducing the \emph{Poincar\'e normal form} of the system using the theory and techniques presented, for example, in books such as \cite{KuBook} or \cite{Wig}. As we shall see below, this is the first instance where the critical exponent $\sigma_c$ comes decisively into play.
\begin{lemma}[Local analysis near the point $P_3$]\label{lem.P3}
For any $\sigma\in(0,\sigma_c)$, the critical point $P_3$ behaves as an attractor for the orbits coming from the half-space $\{X>0\}$ of the phase space associated to the system \eqref{PSsyst}. The orbits entering it contain profiles presenting a tail as $\xi\to\infty$, namely
\begin{equation}\label{beh.P3}
f(\xi)\sim\left(\frac{1}{m-1}\right)^{1/(m-1)}\xi^{-\sigma/(m-1)}, \qquad {\rm as} \ \xi\to\infty.
\end{equation}
For any $\sigma>\sigma_c$, the critical point $P_3$ behaves as a repeller for the orbits coming from the half-space $\{X>0\}$ of the phase space and there are no orbits either entering or going out of it.
\end{lemma}
\begin{proof}
The proof is rather long and will be divided into several steps for the readers' convenience. The steps are similar to the ones in the analysis in dimension $N=1$ \cite[Lemma 3.1]{IS20b} but the calculations give a different outcome depending on $N$ and $\sigma$.

\medskip

\noindent \textbf{Step 1. A new change of variable}. We begin with a change of variable suggested in \cite[Section 3.1F, p.331]{Wig} by letting
\begin{equation}\label{interm3}
v=(m-1)Y+\sigma X, \quad u=\sqrt{m-1}(Z-1), \quad z=X
\end{equation}
to obtain from the system \eqref{PSsyst}, after straightforward calculations, a new system
\begin{equation}\label{interm4}
\left\{\begin{array}{ll}\dot{v}=-\sqrt{m-1}u-\frac{m+1}{2(m-1)}v^2+\frac{K_1(\sigma)}{2(m-1)}zv+\frac{K_2(\sigma)}{m-1}z^2,\\
\dot{u}=\sqrt{m-1}v+uv,\\
\dot{z}=-\frac{\sigma+2}{2}z^2+\frac{1}{2}zv,\end{array}\right.
\end{equation}
where
\begin{equation}\label{coef}
K_1(\sigma)=(3m+1)\sigma-2(m-1)(N-1), \qquad K_2(\sigma)=\sigma[(m-1)(N-2)-m\sigma]
\end{equation}
Since our goal is to establish the first terms of the normal form of the system \eqref{interm4}, we perform one further change of variable according to \cite[Section 8.5]{KuBook} by letting $w=v+iu$, or equivalently
$$
v=\frac{w+\overline{w}}{2}, \quad u=\frac{w-\overline{w}}{2i}.
$$
Starting from $\dot{w}=\dot{v}+i\dot{u}$, using the equations in \eqref{interm4} and the correspondence between $(v,u)$ and $(w,\overline{w})$, we find
\begin{equation}\label{interm5}
\begin{split}
\dot{w}&=i\sqrt{m-1}w+\frac{K_2(\sigma)}{m-1}z^2+\left(\frac{1}{4}-\frac{m+1}{8(m-1)}\right)w^2\\
&-\left(\frac{1}{4}+\frac{m+1}{8(m-1)}\right)\overline{w}^2+\frac{K_1(\sigma)}{4(m-1)}(zw+z\overline{w})-\frac{m+1}{4(m-1)}w\overline{w},
\end{split}
\end{equation}
and
\begin{equation}\label{interm6}
\dot{z}=-\frac{\sigma+2}{2}z^2+\frac{1}{4}zw+\frac{1}{4}z\overline{w}.
\end{equation}

\medskip

\noindent \textbf{Step 2. The Poincar\'e normal form.} Following the general theory in \cite[Section 8.5]{KuBook} and letting $g(z,w,\overline{w})$ (for the equation \eqref{interm6}), respectively $h(z,w,\overline{w})$ (for the equation \eqref{interm5}) be the nonlinear parts of the equations in the previous step, we can derive the Taylor expansions of $g$ and $h$ in the following general forms (according to the notation on \cite[p. 332-333]{KuBook})
$$
g(z,w,\overline{w})=\sum\limits_{j+k+l\geq2}\frac{1}{j!k!l!}g_{jkl}z^jw^k\overline{w}^l, \quad h(z,w,\overline{w})=\sum\limits_{j+k+l\geq2}\frac{1}{j!k!l!}h_{jkl}z^jw^k\overline{w}^l.
$$
In our specific case, one can readily establish the coefficients up to order 2 as it follows:
\begin{equation}\label{interm7}
g_{200}=-(\sigma+2), \ g_{110}=g_{101}=\frac{1}{4}, \ g_{020}=g_{002}=g_{011}=0,
\end{equation}
and (using the notation with $h$ for the $w$ equation)
\begin{equation}\label{interm8}
\begin{split}
&h_{200}=\frac{2K_2(\sigma)}{m-1}, \ h_{020}=\frac{1}{2}-\frac{m+1}{4(m-1)}, \ h_{002}=-\left(\frac{1}{2}+\frac{m+1}{4(m-1)}\right),\\
&h_{110}=h_{101}=\frac{K_1(\sigma)}{4(m-1)}, \ h_{011}=-\frac{m+1}{4(m-1)}.
\end{split}
\end{equation}
We now rely on the explicit formulas given in \cite[Lemma 8.9]{KuBook} to calculate the coefficients of the Poincar\'e normal form of the system having as starting point the coefficients in \eqref{interm7} and \eqref{interm8}. Skipping the technical details (indicated in the particular case $N=1$ in \cite[Lemma 3.1]{IS20b}) and maintaining the terms up to order two, we can write
\begin{equation*}
\left\{\begin{array}{ll}\dot{z}=-\frac{\sigma+2}{2}z^2+O(|(z,w,\overline{w})|^3),\\
\dot{w}=i\sqrt{m-1}w+\frac{K_1(\sigma)}{4(m-1)}zw+O(|(z,w,\overline{w})|^3),\end{array}\right.
\end{equation*}
Undoing the change of variable $w=v+iu$ in order to get back to the variables $(z,v,u)$ and maintaining only the terms up to order two, we finally find the Poincar\'e normal form of the system \eqref{interm4}:
\begin{equation}\label{normal.syst}
\left\{\begin{array}{ll}\dot{z}=-\frac{\sigma+2}{2}z^2+O(|(z,v,u)|^3),\\
\dot{v}=-\sqrt{m-1}u+\frac{K_1(\sigma)}{4(m-1)}zv+O(|(z,v,u)|^3),\\
\dot{u}=\sqrt{m-1}v+\frac{K_1(\sigma)}{4(m-1)}zu+O(|(z,v,u)|^3).\end{array}\right.
\end{equation}
Notice that the coefficient $K_1(\sigma)$ plays an important role in the normal form \eqref{normal.syst} and that $K_1(\sigma)<0$ for $\sigma\in(0,\sigma_c)$ while $K_1(\sigma)>0$ for $\sigma>\sigma_c$.

\medskip

\noindent \textbf{Step 3. End of the proof.} Following again \cite[Section 3.1F]{Wig}, we pass the normal form \eqref{normal.syst} into cylindrical coordinates by letting $v=r\cos\,\theta$, $u=r\sin\,\theta$ and $z$ unchanged. We thus find the following normal form in cylindrical coordinates:
\begin{equation}\label{normal.cyl}
\left\{\begin{array}{ll}\dot{z}=-\frac{\sigma+2}{2}z^2+O(|(z,r)|^3),\\
\dot{r}=\frac{K_1(\sigma)}{4(m-1)}zr+O(|(z,r)|^3),\\
\dot{\theta}=\sqrt{m-1}+O(|(z,r)|^3),\end{array}\right.
\end{equation}
Assume first that we are not considering connections included in the planes $\{z=0\}$ or $\{r=0\}$ and let $\sigma\in(0,\sigma_c)$, which is equivalent to $K_1(\sigma)<0$. We infer then from the first two equations in \eqref{normal.cyl} that in a neighborhood of the critical point $P_3$ (seen as the origin in \eqref{normal.cyl}) both components $z$ and $r$ are monotonically decreasing along the trajectories of the system. This gives that the orbits entering such neighborhood cannot end in a limit cycle and have to enter $P_3$. On the contrary, when $\sigma>\sigma_c$ we have $K_1(\sigma)>0$. We can then integrate (up to third order) the system formed by the first two equations in \eqref{normal.cyl} to get in a neighborhood of the critical point $P_3$ that
\begin{equation*}
\frac{dz}{dr}\sim-\frac{2(\sigma+2)(m-1)}{K_1(\sigma)}\frac{z}{r}=K_3(\sigma)\frac{z}{r},
\end{equation*}
where $K_3(\sigma)<0$ for $\sigma>\sigma_c$, whence we obtain by integration that
\begin{equation}\label{interm9bis}
z\sim Cr^{K_3(\sigma)}, \qquad K_3(\sigma)=-\frac{2(\sigma+2)(m-1)}{K_1(\sigma)}.
\end{equation}
We infer from \eqref{interm9bis} that the orbits in a neighborhood of the critical point $P_3$ and corresponding to integration constants $C>0$ in \eqref{interm9bis} tend to hyperbolas that do not go out of the point. There might be still an orbit going out of $P_3$ corresponding to the constant $C=0$ in \eqref{interm9bis}, that is, tangent to the plane $\{z=0\}$ but not contained in it. But on such an orbit going out of $P_3$ we have $X(\xi)\to0$, $Z(\xi)\to1$, thus we deduce from \eqref{PSchange} that
\begin{equation}\label{interm9}
\frac{Z(\xi)}{X(\xi)^2}=\frac{1}{m}\xi^{\sigma+2}\to+\infty,
\end{equation}
which shows that there is no orbit containing profiles and going out of this point into the region $\{X>0\}$ of the phase space. Notice at this point that also in the case $\sigma\in(0,\sigma_c)$, where we already know that the orbits coming from the interior of the phase space enter $P_3$, we obtain in \eqref{interm9bis} the form in which such orbits enter $P_3$ (in this case with $K_3(\sigma)>0$). It remains to consider the orbits fully contained in the planes $\{z=0\}$ or $\{r=0\}$ of the system \eqref{normal.cyl}. Since $z=X$, the former orbits are contained in the invariant plane $\{X=0\}$ and we discard them, as they do not contain interesting profiles. For the latter plane, $r=0$ implies $v=u=0$, thus such orbits are contained in the line of equation $Z=1$, $(m-1)Y+\sigma X=0$, which in terms of profiles becomes the hyperbola
\begin{equation}\label{hyp}
f(\xi)=\left(\frac{1}{m-1}\right)^{1/(m-1)}\xi^{-\sigma/(m-1)},
\end{equation}
and this is not a solution to Eq. \eqref{ODE}. Thus, no interesting trajectories come from these exceptional planes. Coming back to the previous analysis, we conclude that the critical point $P_3$ behaves as a repeller for $\sigma>\sigma_c$ and as an attractor for the orbits coming from the region $\{X>0\}$ for $\sigma\in(0,\sigma_c)$. In the latter case, the orbits entering $P_3$ have $Z(\xi)\to1$, which is equivalent to \eqref{beh.P3} taking into account the definition of $Z$ in \eqref{PSchange}. To end the proof, the fact that \eqref{beh.P3} holds true as $\xi\to\infty$ follows immediately from \eqref{interm9}. The proof is complete.
\end{proof}
Let us remark here that we omit the local analysis near the point $P_3$ exactly for $\sigma=\sigma_c$. Since $K_1(\sigma)=0$ in this case, the analysis becomes much more involved, as terms of higher order in the Poincar\'e normal form have to be considered. However, the good news is that the analysis of this critical case is not needed for the classification of the profiles, thus we refrain from entering it.

\subsection{Periodic orbits}\label{subsec.cycle}

Apart from the critical points in the plane analyzed above, the system \eqref{PSsyst} presents some explicit periodic orbits (or cycles) lying inside the invariant plane $\{X=0\}$. Indeed, by letting $X=0$ in \eqref{PSsyst}, we obtain the same system as in dimension $N=1$, namely
\begin{equation}\label{syst.X0}
\left\{\begin{array}{ll}\dot{Y}=-\frac{m+1}{2}Y^2+1-Z,\\ \dot{Z}=(m-1)YZ,\end{array}\right.
\end{equation}
that can be integrated to obtain the following curves
\begin{equation}\label{cycles}
Y^2=\frac{2}{m+1}-\frac{1}{m}Z-KZ^{-(m+1)/(m-1)},
\end{equation}
which are periodic orbits of the system provided $K>0$ (as for $K<0$ they cross the plane $\{Z=0\}$ and we are not interested in them). In particular, the one with $K=0$ and the parabolic cylinder in direction $X>0$ constructed from it proved to be very important as a separatrix in the analysis in dimension $N=1$ \cite[Section 4]{IS20b}. These periodic orbits are plotted in Figure \ref{fig1}, where those corresponding to constants $K>0$ lie in the region bounded by the explicit separatrix with $K=0$.

\begin{figure}[ht!]
  \begin{center}
  \includegraphics[width=10cm,height=7.5cm]{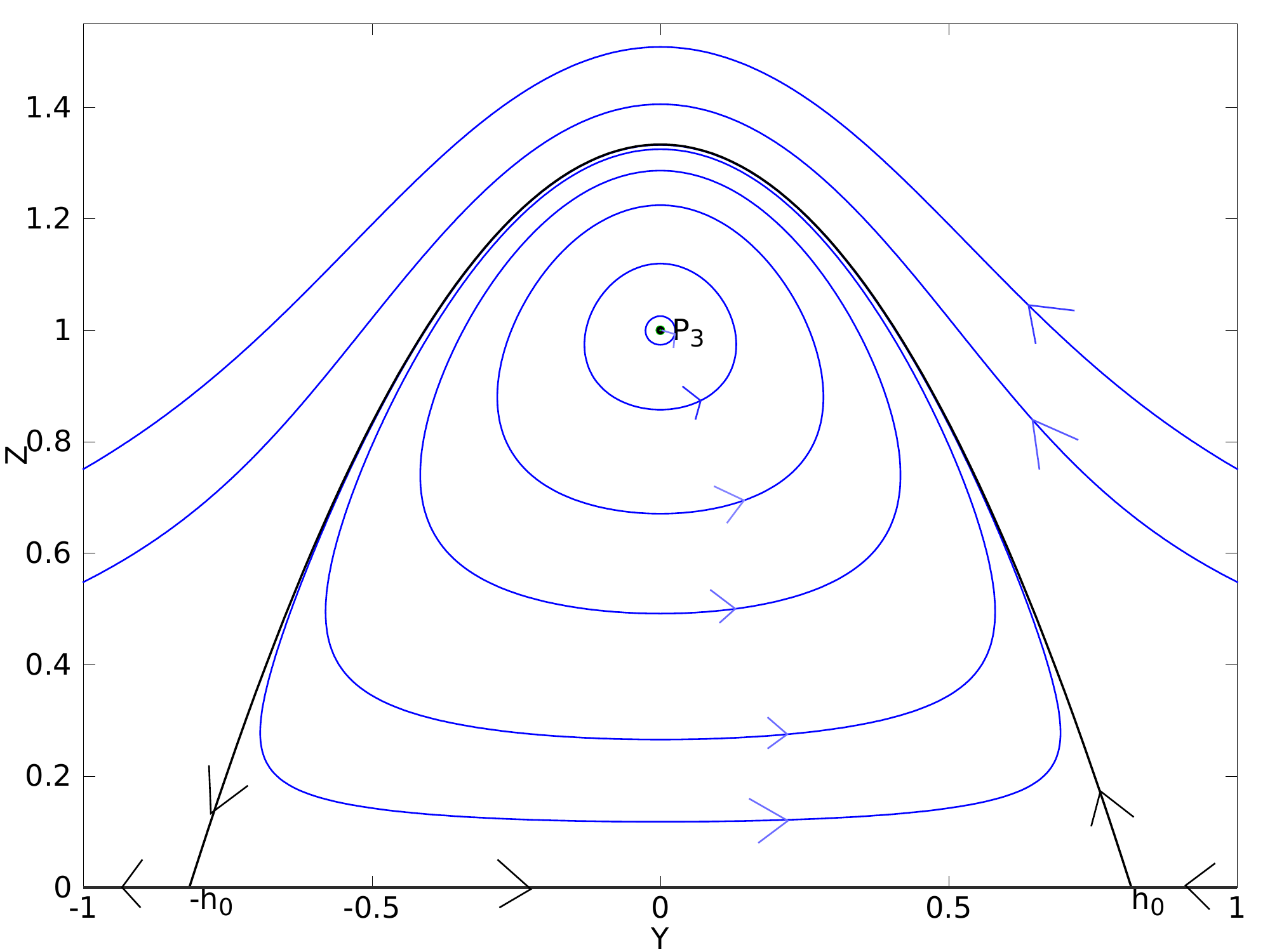}
  \end{center}
  \caption{Periodic orbits in the invariant plane $\{X=0\}$}\label{fig1}
\end{figure} 

In our case, we need to know whether such periodic orbits can be limit cycles for orbits coming from the positive part $\{X>0\}$ of the phase space. Since we are dealing with a three-dimensional dynamical system, the theory of the stability of limit cycles is quite complicated in practice and the analysis of the Poincar\'e map (which is usually the criterion for stability) appears to be not easily available if the analytic form of the periodic orbit is difficult to handle. Thus, we employ a different technique, working directly with the differential equation \eqref{ODE}, to show that at least when $\sigma\neq\sigma_c$ there are no orbits coming from $\{X>0\}$ and having \eqref{cycles} as limit cycles. The idea is to show that, if a solution to \eqref{ODE} has infinitely many oscillations, they have to be damped and their amplitude converges to zero.
\begin{proposition}\label{prop.cycles}
For any $\sigma\neq\sigma_c$, there is no orbit coming from the region $\{X>0\}$ of the phase space associated to the system \eqref{PSsyst} and having any of the orbits in \eqref{cycles} with $K>0$ as limit cycle.
\end{proposition}
\begin{proof}
Assume for contradiction that there exists such an orbit. It is then obvious that, for $X>0$ but very close to zero, it has to oscillate between two fixed values of $Z$ (the maximal and minimal height of the periodic orbit in \eqref{cycles}) having $Z=1$ in the middle. We can thus say that the profiles contained in this orbit oscillate around the hyperbola $Z=1$, that is \eqref{hyp}, and with an almost constant amplitude for $\xi$ large. We will next show that such solutions to Eq. \eqref{ODE} cannot exist, by working directly with the differential equation. To this end, we first introduce a new change of variables given by
\begin{equation}\label{change.cycles}
G(\zeta)=(m-1)^{1/(m-1)}\xi^{\sigma/(m-1)}f(\xi)=Z(\xi)^{1/(m-1)}, \qquad \zeta=\frac{2}{\sigma+2}\xi^{(\sigma+2)/2},
\end{equation}
and derive by straightforward calculations the differential equation solved by $G(\zeta)$, that is (written in the same form as \eqref{ODE})
\begin{equation}\label{ODE2}
(G^m)''(\zeta)+\frac{\overline{N}-1}{\zeta}(G^m)'(\zeta)+G^m(\zeta)-G(\zeta)+\frac{K(\sigma)G^m(\zeta)}{\zeta^2}=0,
\end{equation}
where
$$
\overline{N}=1-\frac{K_1(\sigma)}{(m-1)(\sigma+2)}, \qquad K(\sigma)=-\frac{4m\sigma K_2(\sigma)}{(m-1)^2(\sigma+2)^2}.
$$
and $K_1(\sigma)$, $K_2(\sigma)$ are defined in \eqref{coef}. Let us notice that at least formally, the oscillations of the profile $f(\xi)$ around the hyperbola \eqref{hyp} translate into oscillations of $G(\zeta)$ with respect to the horizontal line $G=1$, which is an equilibrium point of the equation. It is thus enough to prove that such oscillations must be damped and that any positive solution $G(\zeta)$ to \eqref{ODE2} satisfies $\lim\limits_{\zeta\to\infty}G(\zeta)=1$. This is rather expected at a formal level by removing the last term in the limit, but the rigorous proof is more involved and borrows ideas from the proof of \cite[Lemma 1, Section 4.1.2, p. 183-184]{S4}. We begin by multiplying Eq. \eqref{ODE2} by $(G^{m-1}G')(\zeta)$ to get (quitting the dependence on $\zeta$ to simplify the notation)
\begin{equation*}
\frac{1}{m}(G^m)''(G^m)'+\frac{\overline{N}-1}{m\zeta}[(G^m)']^2+(G^{2m-1}-G^m)G'+\frac{K(\sigma)}{\zeta^2}G^{2m-1}G'=0,
\end{equation*}
which integrated on a generic interval $[\zeta_1,\zeta_2]$ gives
\begin{equation}\label{interm11}
\begin{split}
\frac{1}{2m}\left[(G^m)'^2(\zeta_2)-(G^m)'^2(\zeta_1)\right]&+\Phi(\zeta_2)-\Phi(\zeta_1)+\frac{\overline{N}-1}{m}\int_{\zeta_1}^{\zeta_2}\frac{(G^m)'^2(\zeta)}{\zeta}d\zeta\\
&=-K(\sigma)\int_{\zeta_1}^{\zeta_2}\frac{(G^{2m})'(\zeta)}{2m\zeta^2}d\zeta,
\end{split}
\end{equation}
where
\begin{equation*}
\Phi(\zeta)=\frac{1}{2m}G^{2m}(\zeta)-\frac{1}{m+1}G^{m+1}(\zeta).
\end{equation*}
Since we assumed that $G$ has infinitely many oscillations around the constant value $G=1$, it thus has a sequence of local maxima and minima, that we denote by $\zeta_{n}^{M}$, respectively $\zeta_n^{m}$, such that $\zeta_n^M\to\infty$ and $\zeta_n^m\to\infty$ as $n\to\infty$. At any of these points, $(G^m)'(\zeta_n^m)=(G^m)'(\zeta_n^M)=0$. Moreover, taking into account that the orbit containing $G$ approaches in the limit one periodic orbit of the form \eqref{cycles} with $K>0$ and that $Z=G^{m-1}$, we get that $\Phi(\zeta)$ is uniformly bounded for $\zeta$ sufficiently large. We thus infer that, if we let $\zeta_2=\zeta_k^M$ and $\zeta_1=\zeta_j^m$ (that is, one maximum and one minimum point, not necessarily the same ones in the sequence of them) in \eqref{interm11}, the first four terms either directly vanish or are uniformly bounded. We are thus left with the two integral terms. To estimate them, we express first the derivatives $(G^m)'$ and $(G^{2m})'$ (derivatives being taken with respect to $\zeta$) in terms of the component $Z$ in the phase space. Taking into account that $G=Z^{1/(m-1)}$, the fact that in the system \eqref{PSsyst} the derivatives are taken with respect to the independent variable $\eta$ introduced in \eqref{change2}, and the change of variable from $\xi$ to $\zeta$, we have
\begin{equation}\label{interm13}
\begin{split}
\frac{d(G^m)}{d\zeta}&=\frac{d}{d\xi}Z^{m/(m-1)}\frac{d\xi}{d\zeta}=\xi^{-\sigma/2}\frac{d}{d\xi}Z^{m/(m-1)}\\
&=\xi^{-\sigma/2}f(\xi)^{-(m-1)/2}(\xi)\frac{d}{d\eta}Z^{m/(m-1)}\\&=\frac{m}{m-1}\xi^{-\sigma/2}\left[\frac{1}{m-1}\xi^{-\sigma}Z\right]^{-1/2}Z^{1/(m-1)}\dot{Z}\\
&=\frac{m}{\sqrt{m-1}}Z^{1/(m-1)+1/2}[(m-1)Y+\sigma X],
\end{split}
\end{equation}
where in the last calculation step we used the equation for $\dot{Z}$ from the system \eqref{PSsyst}. Moreover, following a similar calculations as in \eqref{interm13}, we also get
\begin{equation}\label{interm14}
\frac{d(G^{2m})}{d\zeta}=\frac{2m}{\sqrt{m-1}}Z^{(m+1)/(m-1)+1/2}[(m-1)Y+\sigma X].
\end{equation}
We deduce from \eqref{interm14} and the fact that the trajectory containing the profile $f$ approaches a limit cycle of the form \eqref{cycles} that $(G^{2m})'$ is uniformly bounded by some constant $L>0$ for $\zeta$ sufficiently large, thus
\begin{equation}\label{interm15}
\left|\int_{\zeta_1}^{\infty}\frac{(G^{2m})'(\zeta)}{2m\zeta^2}d\zeta\right|\leq\frac{L}{2m}\int_{\zeta_1}^{\infty}\frac{1}{\zeta^2}d\zeta<\infty.
\end{equation}
On the contrary, \eqref{interm13} together with the fact that $G=Z^{1/(m-1)}$ and that the trajectory containing the profile $f$ defining our $G$ in \eqref{change.cycles} approaches a limit cycle imply that $G^m$ (hence also $(G^m)'$) approaches a periodic function for $\zeta$ sufficiently large (as being a solution to an autonomous equation in the limit). This gives that the remaining integral in \eqref{interm11}, namely
\begin{equation}\label{interm16}
\frac{\overline{N}-1}{m}\int_{\zeta_1}^{\infty}\frac{(G^m)'^2(\zeta)}{\zeta}d\zeta
\end{equation}
is divergent. We thus reach a contradiction in \eqref{interm11} by choosing $\xi_1$ a fixed minimum point, $\xi_2=\xi_{n}^M$ and passing to the limit as $n\to\infty$, as all but one term are bounded, provided that the coefficient in front of the integral in \eqref{interm16} is not equal to zero. Recalling that
$$
\frac{\overline{N}-1}{m}=-\frac{K_1(\sigma)}{m(m-1)(\sigma+2)},
$$
and the fact that $K_1(\sigma)=0$ if and only if $\sigma=\sigma_c$, the contradiction is reached for any $\sigma\neq\sigma_c$, as claimed.
\end{proof}
We again refrain from analyzing the critical case $\sigma=\sigma_c$ here, as it will be not needed in the sequel. Let us notice that the above proof also applies for $\sigma=0$, providing a slight improvement of the result in \cite[Remark, p. 184]{S4}, where it is only shown that the oscillations are damped but not convergent to zero. The proof of Proposition \ref{prop.cycles} completes the local analysis of the finite part of the plane. We are now ready to pass to the analysis of the critical points at infinity.

\section{Analysis at infinity of the phase space}\label{sec.infty}

The aim of this section is to complete the local analysis of the phase space associated to the system \eqref{PSsyst} by studying the critical points at infinity. We follow the recipe in \cite[Section 3.10]{Pe} where new variables $(\overline{X},\overline{Y},\overline{Z},W)$ are introduced in order to pass to the Poincar\'e hypersphere. More precisely, we let
$$
X=\frac{\overline{X}}{W}, \ Y=\frac{\overline{Y}}{W}, \ Z=\frac{\overline{Z}}{W},
$$
and following \cite[Theorem 4, Section 3.10]{Pe}, the critical points at infinity of the phase space associated to the system \eqref{PSsyst} lie on the equator of the Poincar\'e hypersphere, that is, at points of the form $(\overline{X},\overline{Y},\overline{Z},0)$ such that $\overline{X}^2+\overline{Y}^2+\overline{Z}^2=1$, and are at the same time solutions to the following system:
\begin{equation}\label{Poincare1}
\left\{\begin{array}{ll}\overline{X}Q_2(\overline{X},\overline{Y},\overline{Z})-\overline{Y}P_2(\overline{X},\overline{Y},\overline{Z})=0,\\
\overline{X}R_2(\overline{X},\overline{Y},\overline{Z})-\overline{Z}P_2(\overline{X},\overline{Y},\overline{Z})=0,\\
\overline{Y}R_2(\overline{X},\overline{Y},\overline{Z})-\overline{Z}Q_2(\overline{X},\overline{Y},\overline{Z})=0,\end{array}\right.
\end{equation}
where $P_2$, $Q_2$ and $R_2$ are the homogeneous second degree parts in the right hand side of the system \eqref{PSsyst}, namely
\begin{equation*}
\begin{split}
&P_2(\overline{X},\overline{Y},\overline{Z})=\frac{m-1}{2}\overline{X}\overline{Y}-\overline{X}^2,\\
&Q_2(\overline{X},\overline{Y},\overline{Z})=-\frac{m+1}{2}\overline{Y}^2-(N-1)\overline{X}\overline{Y},\\
&R_2(\overline{X},\overline{Y},\overline{Z})=\overline{Z}((m-1)\overline{Y}+\sigma\overline{X}).
\end{split}
\end{equation*}
The system \eqref{Poincare1} becomes
\begin{equation}\label{Poincare2}
\left\{\begin{array}{ll}\overline{X}\overline{Y}((2-N)\overline{X}-m\overline{Y})=0,\\
\overline{X}\overline{Z}\left((\sigma+1)\overline{X}+\frac{m-1}{2}\overline{Y}\right)=0,\\
\overline{Y}\overline{Z}\left((\sigma+N-1)\overline{X}+\frac{3m-1}{2}\overline{Y}\right)=0,\end{array}\right.
\end{equation}
and taking into account that we are only considering the part of the equator where $\overline{X}\geq0$ and $\overline{Z}\geq0$, we readily get the following critical points:
\begin{equation*}
\begin{split}
&Q_1=(1,0,0,0), \ \ Q_{2,3}=(0,\pm1,0,0), \ \ Q_4=(0,0,1,0), \\
& {\rm and} \ Q_5=\left(\frac{m}{\sqrt{(2-N)^2+m^2}},\frac{2-N}{\sqrt{(2-N)^2+m^2}},0,0\right).
\end{split}
\end{equation*}
Let us notice at this point that the dimension $N=2$ appears to be critical in the analysis of these points. Thus, we will let it aside for a moment and restrict ourselves first at dimensions $N\geq3$. Despite the fact that the differences with respect to the dimension $N=1$ (studied in \cite{IS20b}) appear to be not significant in the expression of the critical points and system \eqref{Poincare2}, we shall see that they are noticeable with respect to the qualitative behavior of some of the profiles. We analyze below the critical points one by one.
\begin{lemma}[Analysis of the point $Q_1$]\label{lem.Q1}
For $N\geq3$, the critical point at infinity represented as $Q_1=(1,0,0,0)$ in the Poincar\'e hypersphere has a two-dimensional unstable manifold and a one-dimensional stable manifold. The orbits going out of this point to the finite part of the phase space contain profiles $f(\xi)$ such that $f(0)=a>0$ and $f'(0)=0$, corresponding to the assumption (P1) in Definition \ref{def1}.
\end{lemma}
\begin{proof}
Following part (a) of \cite[Theorem 5, Section 3.10]{Pe}, the flow of the system \eqref{PSsyst} in a neighborhood of $Q_1$ is topologically equivalent to the flow near the origin for the new system
\begin{equation}\label{systinf1}
\left\{\begin{array}{ll}-\dot{y}=(N-2)y+my^2+zw-w^2,\\
-\dot{z}=-(\sigma+1)z-\frac{m-1}{2}yz,\\
-\dot{w}=-w+\frac{m-1}{2}yw,\end{array}\right.
\end{equation}
where we choose the minus sign in the system \eqref{systinf1} in order to match the direction of the flow given, for example, by the first equation of the original system \eqref{PSsyst},
$$
\dot{X}=\frac{1}{2}X[(m-1)Y-2X],
$$
leading to $\dot{X}<0$ in a neighborhood of $Q_1$, taking into account that $|X/Y|\to+\infty$ near this point. Since $N\geq3$, the two-dimensional unstable manifold and one-dimensional stable manifold are obvious. The profiles going out of $Q_1$ on the two-dimensional unstable manifold are contained in orbits tangent to the plane $\{y=0\}$ in the system \eqref{systinf1}. On the one hand, we get from \eqref{systinf1} that
$$
\frac{dz}{dw}\sim(\sigma+1)\frac{z}{w},
$$
or equivalenty $z\sim Cw^{\sigma+1}$, which in variables $(X,Z)$ writes
$$
\frac{Z}{X}\sim\frac{C}{X^{\sigma+1}},
$$
whence $X^{\sigma}Z\sim C$, that is $f(\xi)\sim C$. On the other hand, from the tangency of the unstable manifold, we do not only get that $y\to0$ but also $y/z\to0$, which implies $Y/Z\to0$. Together with the fact that $Y/X\to0$ on these profiles, we simultaneously get that
$$
\frac{Y}{X}=\frac{\xi f'(\xi)}{f(\xi)}\to0, \qquad \frac{Y}{Z}=\sqrt{\frac{m}{m-1}}\xi^{-\sigma}f(\xi)^{-(m+1)/2}f'(\xi)\to0,
$$
which, together with the fact that $f(\xi)\sim C$, readily gives that $f'(\xi)\to0$. The fact that all the limits and equivalences above are taken as $\xi\to0$ follows immediately from the fact that $X(\xi)\to\infty$, which leads to $\xi^{-1}\to\infty$.
\end{proof}
\begin{lemma}[Analysis of the points $Q_2$ and $Q_3$]\label{lem.Q23}
For $N\geq2$, the critical points at infinity represented as $Q_{2,3}=(0,\pm1,0,0)$ in the Poincar\'e hypersphere are an unstable node, respectively a stable node. The orbits going out of $Q_2$ contain profiles $f(\xi)$ such that there exists $\xi_0\in(0,\infty)$ with $f(\xi_0)=0$, $f'(\xi_0)=+\infty$, while the orbits entering the point $Q_3$ and contain profiles $f(\xi)$ such that there exists $\xi_0\in(0,\infty)$ with $f(\xi_0)=0$, $f'(\xi_0)=-\infty$.
\end{lemma}
\begin{proof}
We employ this time part (b) of \cite[Theorem 5, Section 3.10]{Pe} to infer that the flow of the system near the points $Q_2$ and $Q_3$ is topologically equivalent to the flow near the origin for the system
\begin{equation}\label{systinf2}
\left\{\begin{array}{ll}\pm\dot{x}=-mx+(2-N)x^2+xw^2-xzw,\\
\pm\dot{z}=-\frac{3m-1}{2}z-(\sigma+N-1)xz-z^2w+zw^2,\\
\pm\dot{w}=-\frac{m+1}{2}w-zw^2+w^3-(N-1)xw,\end{array}\right.
\end{equation}
where the minus sign corresponds to the point with $Y\to+\infty$, that is $Q_2$ and the plus sign corresponds to the point with $Y\to-\infty$, that is $Q_3$. This is seen from the fact that both points are characterized by $|Y/X|\to\infty$ and $|Y/Z|\to\infty$ in a neighborhood of them, thus in the second equation of the system \eqref{PSsyst} the term $-(m+1)Y^2/2$ dominates near $Q_2$ and $Q_3$, indicating a direction of the flow from right to left ($Y$ decreasing along the trajectories) near the two points. Since there is no influence of $N$ in the linear terms in the system \eqref{systinf2}, the rest of the analysis is totally identical to the one performed in dimension $N=1$ in \cite[Lemma 2.4]{IS20b} to which we refer.
\end{proof}
We are now left with the points $Q_4$ and $Q_5$. As we shall see, the analysis of these points is much more involved than the analogous one in dimension $N=1$, since in fact they \emph{alternate} to be unstable nodes and, depending on the value of $\sigma>0$, pass from one to the other the orbits presenting a specific vertical asymptote at $\xi=0$. This is made precise in the following statement.
\begin{lemma}[Analysis of the points $Q_4$ and $Q_5$]\label{lem.Q4Q5}
Let $N\geq3$. Then, depending on the value of $\sigma$, either $Q_5$ is an unstable node or $Q_4$ is an unstable node. In any of the two cases, the orbits going out of the point that is an unstable node contain profiles presenting a vertical asymptote at $\xi=0$ with behavior
\begin{equation}\label{beh.asympt}
f(\xi)\sim C\xi^{(2-N)/m}, \qquad {\rm as} \ \xi\to0, \ C>0 \ {\rm free \ constant},
\end{equation}
while the orbits going out of the point which is not an unstable node are included in the infinity part of the phase space and do not contain other profiles.
\end{lemma}
\begin{proof}
Since the proof is longer and more involved than the previous ones, it will be divided into several steps for the easiness of the reading.

\medskip

\noindent \textbf{Step 1. Local analysis near $Q_5$}. We can analyze locally $Q_5$ by identifying it with the critical point $(y,z,w)=((2-N)/m,0,0)$ in the system \eqref{systinf1}. The linearization of \eqref{systinf1} near $Q_5$ has the matrix
$$
M(Q_5)=\left(
         \begin{array}{ccc}
           N-2 & 0 & 0 \\
           0 & \frac{2m(\sigma+1)-(N-2)(m-1)}{2m} & 0 \\
           0 & 0 & \frac{2m+(N-2)(m-1)}{2m} \\
         \end{array}
       \right),
$$
thus $Q_5$ is an unstable node if
\begin{equation}\label{interm17}
K(m,N,\sigma):=2m(\sigma+1)-(N-2)(m-1)>0
\end{equation}
Assume now that \eqref{interm17} holds true. Then the orbits going out of $Q_5$ are characterized by $y=(2-N)/m$, that is $Y/X\sim(2-N)/m$, which in terms of profiles reads
$$
\left(f(\xi)^{(m-1)/2}\right)'\sim\frac{(2-N)(m-1)}{2m}\frac{f(\xi)^{(m-1)/2}}{\xi},
$$
and this leads to the desired behavior \eqref{beh.asympt} by integration. Since $X/Z\to\infty$ in a neighborhood of $Q_5$, we get that
$$
\xi^{-1-\sigma}f(\xi)^{-(m-1)/2}\to\infty,
$$
and we infer from \eqref{beh.asympt} that necessarily $\xi^{-1-\sigma}\to\infty$, which gives that the previous local behavior must be taken as $\xi\to0$. When \eqref{interm17} does not hold true and we have the contrary inequality, the unstable manifold remains two-dimensional and it is completely contained in the invariant plane $\{z=0\}$ of the system \eqref{systinf1}, which in terms of profiles reduces to either $Z=0$ or $X=\infty$, thus in both cases no new profiles exist and $Q_5$ behaves like a saddle for the orbits approaching it in the finite part of the plane.

\medskip

\noindent \textbf{Step 2. Identification of $Q_5$ and $Q_4$}. We are left with the question of what happens with the profiles with behavior as in \eqref{beh.asympt} when \eqref{interm17} is no longer true, and if the remaining point $Q_4$ does not bring any other new profiles. Moreover, the point $Q_4$ cannot be easily analyzed using the recipes in \cite[Section 3.10]{Pe}. To overcome these difficulties, we introduce a different change of variable
\begin{equation}\label{change3}
x=\frac{1}{m(m-1)}\xi^2f(\xi)^{1-m}, \ y=\frac{\xi f'(\xi)}{f(\xi)}, \ z=\frac{\xi^{\sigma+2}}{m}, \ \frac{d}{d\eta}=\xi\frac{d}{d\xi},
\end{equation}
and in these dependent variables $(x,y,z)$ and independent variable $\eta$ Eq. \eqref{ODE} transforms into the system
\begin{equation}\label{PSsyst3}
\left\{\begin{array}{ll}\dot{x}=x(2-(m-1)y),\\
\dot{y}=-my^2-(N-2)y+x-z,\\
\dot{z}=(\sigma+2)z,\end{array}\right.
\end{equation}
having (among others) the critical point $Q_4'=(0,(2-N)/m,0)$. It is straightforward to check that the linearization of the system \eqref{PSsyst3} near the point $Q_4'$ has a matrix with eigenvalues
$$
\lambda_1=\frac{2m+(N-2)(m-1)}{2m}, \ \lambda_2=N-2, \ \lambda_3=\sigma+2,
$$
thus $Q_4'$ is an unstable node. The orbits going out of $Q_4'$ contain profiles with the properties $z\to0$ and $y\to(2-N)/m$, that readily give \eqref{beh.asympt} as $\xi\to0$, thus the orbits going out of $Q_5$ in Step 1 (when \eqref{interm17} was true) are identified with the orbits going out of $Q_4'$. Notice next that, if we look for profiles with behavior given by \eqref{beh.asympt}, that is, going out of $Q_4'$, we compute
$$
\frac{Z}{X}=\sqrt{\frac{m-1}{m}}\xi^{\sigma+1}f(\xi)^{(m-1)/2}\sim K\xi^{K(m,N,\sigma)/2m}, \ {\rm as} \ \xi\to0,
$$
with $K(m,N,\sigma)$ given in \eqref{interm17}. It thus follows that, when $K(m,N,\sigma)<0$ we get $Z/X\to+\infty$ as $\xi\to0$, and in the same way $Z/Y\to-\infty$. In this case the unstable node $Q_4'$ is identified with the unique point at infinity where $Z$ dominates over $X$ and $Y$ among the critical points of the system \eqref{PSsyst}, that is $Q_4$. We thus notice that, when the sign of $K(m,N,\sigma)$ changes, the orbits containing profiles with behavior given by \eqref{beh.asympt} \emph{immigrate from $Q_5$ to $Q_4$}.

\medskip

\noindent \textbf{Step 3. No other orbits go out of $Q_4$}. It remains the final question of whether there are other orbits \emph{specific to $Q_4$}, with a different behavior than \eqref{beh.asympt} and that are not seen in the change of variable \eqref{change3}. Such orbits are characterized by $Z\to\infty$ and $Z/X\to\infty$, together with $|Z/Y|\to\infty$. We go back then to Eq. \eqref{ODE} and notice first that
$$
\frac{\xi^{\sigma}f(\xi)^m}{f(\xi)}=Z(\xi)\to\infty,
$$
thus the term $f(\xi)/(m-1)$ is of lower order than the last term of \eqref{ODE}. Moreover, $Z(\xi)\to\infty$ implies that $f(\xi)\to\infty$, thus the orbits going out of $Q_4$ must contain profiles with a vertical asymptote. In particular, also $f^m$ has a vertical asymptote at $\xi=0$, which implies in particular that $\ln\,f^m$ has a vertical asymptote too, hence
$$
\frac{(f^m)'}{f^m}=(\ln\,f^m)'\to-\infty,
$$
and this gives that the term $(N-1)(f^m)'/\xi$ dominates over the term $\xi^{\sigma}f(\xi)^m$ in Eq. \eqref{ODE}. In a first order of approximation, we are thus left with the first two terms of Eq. \eqref{ODE} (the two ones involving derivatives), and an integration of the equation formed with them readily leads to the behavior \eqref{beh.asympt}. It thus follows that all the orbits going out of $Q_4$ contain profiles satisfying \eqref{beh.asympt}, which were already found as orbits going out of $Q_4'$ in the system \eqref{PSsyst3}.
\end{proof}
We conclude that, in fact, the limitation introduced by the sign of $K(m,N,\sigma)$ is purely technical and in fact the critical points $Q_5$ and $Q_4$ can be seen as a single one with respect to the analysis of the orbits going out of them. In order to simplify the notation, we \textbf{make the convention to call $Q_5$} throughout the paper this unstable node containing the profiles with behavior given by \eqref{beh.asympt}. We close this section with the analysis of the points $Q_1$ and $Q_5$ in dimension $N=2$.
\begin{lemma}[Local analysis of the critical point $Q_1=Q_5$ in dimension $N=2$]\label{lem.Q1Q5N2}
Let $N=2$. Then the critical points $Q_1$ and $Q_5$ coincide and the new critical point (that we relabel $Q_1$) is a saddle-node in the sense of the theory in \cite[Section 3.4]{GH}. There exists a two-dimensional unstable manifold on which the orbits go out tangent to the plane $\{Y=0\}$ and these orbits contain good profiles with $f(0)=A>0$, $f'(0)=0$, satisfying thus assumption (P1) in Definition \ref{def1}. All the rest of the orbits go out into the region $\{Y<0\}$ and contain profiles with the local behavior
\begin{equation}\label{beh.Q12}
f(\xi)\sim D\left(-\ln\,\xi\right)^{1/m}, \qquad {\rm as} \ \xi\to0, \ D>0 \ {\rm free \ constant}.
\end{equation}
\end{lemma}
\begin{proof}
We go back again to the system \eqref{systinf1} and notice that exactly for $N=2$, the linear term in the equation for $\dot{y}$ vanishes. In fact, a simple calculation with matrices and eigenvector $(1,0,0)$ gives that, if we let $\mu=2-N$, we are dealing at $\mu=0$ with a \emph{transcritical bifurcation} as in \cite{S73} (see also \cite[Section 3.4]{GH}). It follows that at $N=2$ the critical point $Q_1=Q_5$ is a saddle-node with a matrix having eigenvalues $\lambda_1=0$, $\lambda_2=\sigma+1$ and $\lambda_3=1$. The two-dimensional stable manifold tangent to the vector space spanned by the eigenvectors $e_2=(0,1,0)$ and $e_3=(0,0,1)$, that is $\{y=0\}$, contains orbits inherited from the critical point $Q_1$ and their analysis is totally similar as the one in Lemma \ref{lem.Q1}, leading to good profiles such that $f(0)>0$ and $f'(0)=0$. All the other orbits, according to the theory in \cite[Section 3.4]{GH}, go out tangent to the line $\{y=0\}$ (which is the direction of the eigenvector of the eigenvalue $\lambda_1=0$) in the variables of the system \eqref{systinf1}. This means that $|y/w|\to\infty$, $|y/z|\to\infty$ on these orbits when approaching the point $Q_1$, thus we infer from the first equation in \eqref{systinf1} that $\dot{y}\sim-my^2<0$, which gives that the orbits go out into the region $\{y<0\}$. But with respect to the initial variables  we have $y=Y/X$ and thus the region $\{y<0\}$ coincides with the region $\{Y<0\}$. This also implies that, without absolute values, $y/w\to-\infty$ and $y/x\to-\infty$ along these orbits. It remains to establish that the profiles contained in these orbits have a local behavior given by \eqref{beh.Q12}. To this end, we first notice that, in terms of profiles, we have
$$
y=\frac{Y}{X}=\frac{\xi f'(\xi)}{f(\xi)}, \ z=\frac{Z}{X}=\sqrt{\frac{m-1}{m}}\xi^{\sigma+1}f(\xi)^{(m-1)/2}, \ w=\frac{1}{X}=\frac{\xi f(\xi)^{-(m-1)/2}}{\sqrt{m(m-1)}}.
$$
The next plan is to go directly to Eq. \eqref{ODE} and show that in a neighborhood of the point $Q_1$ and along these orbits, the first two terms (the ones including derivatives) dominate over the last two and the local behavior is given in a first order approximation by the combination of them. We have
\begin{equation}\label{interm22}
\frac{(f^m)'(\xi)/\xi}{\xi^{\sigma}f(\xi)^m}=\frac{mf'(\xi)}{\xi^{\sigma+1}f(\xi)}=C(m)\frac{y}{zw}\to-\infty
\end{equation}
and
\begin{equation}\label{interm23}
\frac{(f^m)'(\xi)/\xi}{f(\xi)/(m-1)}=\frac{m(m-1)f(\xi)^{m-1}y}{\xi^2}=\frac{y}{w^2}\to-\infty.
\end{equation}
We next infer from \eqref{interm22} and \eqref{interm23} that the last two terms in Eq. \eqref{ODE} are of lower order with respect to the term $(f^m)'(\xi)/\xi$, which is easy to see that is of the same order as the first one. Thus, the local behavior of the profiles contained in the orbits going out tangent to $\{y=0\}$ in a small neighborhood of $Q_1$ is given in first approximation by the joint effect of the first two terms in \eqref{ODE}, that is
$$
(f^m)''(\xi)+\frac{1}{\xi}(f^m)'(\xi)\sim0,
$$
which leads to \eqref{beh.Q12} by a direct integration. It is easy to check that this behavior is taken as $\xi\to0$ by discarding the other possibilities $\xi\to\xi_0\in(0,\infty)$ or $\xi\to\infty$ in the line of similar proofs done in the previous Lemmas, we omit here the details.
\end{proof}
We are now ready to enter the global analysis of the system.

\section{Analysis in the invariant plane $\{Z=0\}$}\label{sec.Z0}

This shorter section presents a few technical results that will be very useful in the forthcoming global analysis. Before going to the study of the whole phase space associated to the system \eqref{PSsyst}, we restrict ourselves to the plane $\{Z=0\}$ and establish how the connections go inside this plane. Such orbits do not contain profiles but they are usually lower limits of the manifolds in the three-dimensional phase space and we have to understand where these connections go. The reduced system in the plane $\{Z=0\}$ writes
\begin{equation}\label{syst.Z0}
\left\{\begin{array}{ll}\dot{X}=\frac{m-1}{2}XY-X^2,\\
\dot{Y}=-\frac{m+1}{2}Y^2+1-(N-1)XY,\end{array}\right.
\end{equation}
with the same critical points as in Sections \ref{sec.local} and \ref{sec.infty} except for $Q_4$. Also notice here that the analysis performed in Lemma \ref{lem.P2} gives that $P_2$ is an attractor in the plane $\{Z=0\}$, while $P_0$ and $P_1$ are saddle points. Moreover, $Q_1$ is also a saddle point and $Q_5$ is always an unstable node in the plane $\{Z=0\}$ provided $N\geq3$, as it follows from Lemma \ref{lem.Q1} and by analyzing the two-dimensional unstable manifold in Lemma \ref{lem.Q4Q5}. We then have
\begin{lemma}\label{lem.P1Z0}
The unique orbit entering the critical point $P_1$ in the invariant plane $\{Z=0\}$ comes from the node $Q_5$.
\end{lemma}
\begin{proof}
The linearization of the system \eqref{syst.Z0} near the point $P_1$ has the matrix
$$
M_0(P_1)=\left(
           \begin{array}{cc}
             -\frac{m-1}{2}h_0 & 0 \\
             (N-1)h_0 & (m+1)h_0 \\
           \end{array}
         \right),
$$
thus the unique orbit entering the saddle point $P_1$ enters tangent to the eigenvector $((3m+1)/2,-(N-1))$ corresponding to the negative eigenvalue of $M_0(P_1)$. Since $-(N-1)\leq-1<-h_0$, it follows that the orbit enters from the region $\{Y<-h_0\}$. We next consider the line $\{Y=-h_0\}$ and the flow of the system \eqref{syst.Z0} on this line is given by the sign of the expression $(N-1)h_0X$, which is positive. This proves that the line $\{Y=-h_0\}$ cannot be crossed from right to left, thus the orbit stays forever in the region $\{Y<-h_0\}$ and it must come from $Q_5$.
\end{proof}
We will now inspect the (also unique) orbits going out of $P_0$ and $Q_1$, which are also saddle points when restricted to the plane $\{Z=0\}$. We will prove below that they connect to the attractor $P_2=(X(P_2),Y(P_2))$ with $X(P_2)$, $Y(P_2)$ defined in \eqref{xp2}.
\begin{lemma}\label{lem.P0Q1Z0}
The unique orbit going out of $P_0$ inside the invariant plane $\{Z=0\}$ enters the critical point $P_2$. The same holds true for the unique orbit going out of $Q_1$.
\end{lemma}
\begin{proof}
\textbf{Connection $P_0-P_2$}. We consider the curves (isoclines of the system \eqref{syst.Z0})
$$
C_1: \ Y=\frac{2}{m-1}X, \qquad C_2: \ -\frac{m+1}{2}Y^2+1-(N-1)XY=0,
$$
the first one being the line connecting the origin $(0,0)$ to $P_2$, while the second passes through both $P_0$ and $P_2$. The intersection between the line $C_1$ and the horizontal line $\{Y=h_0\}$ passing through $P_0$ is given by $X=(m-1)h_0/2$. Let us thus consider the rectangle limited by the $X$ and $Y$ axis, the line $\{Y=h_0\}$ and the line $\{X=(m-1)h_0/2\}$. It is easy to check that $Y(P_2)<h_0$ and $X(P_2)<(m-1)h_0/2$ (the equality in both being attained for $N=1$), thus the point $P_2$ lies inside the rectangle. By calculating the directions of the flow of the system \eqref{syst.Z0} on the sides of this rectangle, we find that any orbit entering the rectangle must remain there forever. It is easy to check that the orbit going out of $P_0$ enters the rectangle. Moreover, if we call $F(X,Y)$ the vector field of the system \eqref{syst.Z0}, we have
$$
{\rm div}\,F(X,Y)=-\frac{m+3}{2}Y-(N-1)X<0
$$
inside the rectangle, and Bendixon's criteria \cite[Theorem 1, Section 3.9]{Pe} ensures that no limit cycles can exist inside the rectangle. We then infer from the Poincar\'e-Bendixon's theorem \cite[Section 3.7]{Pe} that the orbit from $P_0$ has to enter a critical point inside the rectangle, and the only possible point is $P_2$, as claimed.

\medskip

\noindent \textbf{Connection $Q_1-P_2$}. It is not difficult to see that the orbit going out of $Q_1$ goes out tangent to the line $\{Y=0\}$ and enters the region $\{Y>0\}$. Thus, it follows that the orbit enters the strip $\{0<Y<h_0\}$ from which it cannot go out. Moreover, using the first equation of the system \eqref{syst.Z0}, we infer that
$$
\dot{X}=X\left(\frac{(m-1)}{2}Y-X\right)<0,
$$
while $Y<h_0$ and $X>(m-1)h_0/2$. This easily proves that the orbit going out of $Q_1$ must enter first the rectangle considered in the previous step and then $P_2$.
\end{proof}
After these preparatory sections, we are in a position to start proving our main results.

\section{Existence of good profiles for $\sigma\in(0,\sigma_c)$}\label{sec.exist}

This section is devoted to the proof of Theorem \ref{th.exist.N4} and also to the existence part in Theorem \ref{th.exist.N23} for the same range of $\sigma$. The proof is involved and technical and follows a number of steps, being in the end based on the "three sets" argument. The main idea is to follow the two-dimensional manifold entering $P_1$. As a convention of notation, we will identify throughout the next sections surfaces in the phase space with the functions whose graph they represent (for example $Z(X,H)$, $Z(X,Y)$ etc.).
\begin{proof}[Proof of Theorem \ref{th.exist.N4}, existence part]
\textbf{Step 1.} In a first step, we show that the orbits entering $P_1$ and lying very close to the plane $\{Z=0\}$ come from the unstable node conventionally named $Q_5$ at the end of Lemma \ref{lem.Q4Q5}. Indeed, Lemma \ref{lem.P1Z0} together with the fact that $Q_5$ is an unstable node (at least in the case when $K(m,N,\sigma)\geq0$ in \eqref{interm17}) give by a standard continuity argument that there are more connections entering $P_1$ and coming from $Q_5$. In the case when $K(m,N,\sigma)<0$, there are orbits entering $P_1$ as close as possible (in a tubular neighborhood) to the one entering $P_1$ and coming from $Q_5$ inside the plane $\{Z=0\}$. But then $Q_5$ acts as a saddle point and rejects the other connections, that will actually come from $Q_4$; the latter follows easily from Lemma \ref{lem.Q4Q5} where we identify $Q_5$ and $Q_4$ with a single unstable node for any value of $\sigma$. The interested reader can rather easily check that the profiles with an interface at $\xi_0\in(0,\epsilon)$ for $\epsilon>0$ sufficiently small are contained in these orbits, by using the relation \eqref{interm9}, but we omit the details, since we will not use this fact later.

\medskip

\noindent \textbf{Step 2. Approximating the stable manifold and constructing a separatrix surface.} Since we want to better analyze the two-dimensional stable manifold of $P_1$, we translate this point at the origin by setting $H=Y+h_0$, where we recall that $h_0=\sqrt{2/(m+1)}<1$. The system \eqref{PSsyst} becomes in these variables
\begin{equation}\label{PSsyst2}
\left\{\begin{array}{ll}\dot{X}=\frac{m-1}{2}XH-\frac{m-1}{2}Xh_0-X^2,\\
\dot{H}=-\frac{m+1}{2}H^2+(m+1)h_0H-Z-(N-1)XH+(N-1)h_0X,\\
\dot{Z}=Z[(m-1)H-(m-1)h_0+\sigma X],\end{array}\right.
\end{equation}
We know that in a neighborhood of the point $P_1$ we have a two-dimensional stable manifold of the form $Z=f(X,H)$ and we look for a local approximation of this manifold up to order two. The algorithm for this local approximation of the stable manifold is justified by the general theory given, for example, in the book \cite[Section 2.7, p. 79-82]{Sh1}. We thus approximate $Z=f(X,H)$ by its Taylor development up to second order near the origin by letting
\begin{equation}\label{manif.P1}
Z_1(X,H)=AX+BH+CX^2+DH^2+EXH+O(|(X,H)|^3)
\end{equation}
with coefficients $A$, $B$, $C$, $D$, $E$ to be determined. Calculating
$$
\dot{Z}=A\dot{X}+2CX\dot{X}+E\dot{X}H+B\dot{H}+2DH\dot{H}+EX\dot{H},
$$
then replacing the expressions of $\dot{Z}$, $\dot{X}$ and $\dot{H}$ with the right hand sides of the equations in the system \eqref{PSsyst2} and identifying the coefficients to similar terms, we obtain
\begin{equation}\label{interm19}
\begin{split}
&A=\frac{4m(N-1)h_0}{3m+1}, \ \ B=2h_0m, \ \ C=\frac{2(N-1)(mN-4m\sigma-6m-N+2}{(3m+1)(5m-1)},\\
&D=-m, \ \ E=-\frac{4m(N+\sigma-1)}{5m-1}.
\end{split}
\end{equation}
With these coefficients in mind, we introduce the following surface, that will be our main tool to separate and drive connections in the phase space:
\begin{equation}\label{surface}
\begin{split}
Z(X,H)&=-mH^2-\frac{(\sigma+2)(2N+\sigma-2)}{8m}X^2-\left(N-1+\frac{\sigma}{2}\right)XH\\
&+\frac{(2N+\sigma-2)h_0}{2}X+2mh_0H,
\end{split}
\end{equation}
describing over the plane $\{X=0\}$ the curve $Z=-mH^2+2mh_0H$, which is nothing else than the unique curve from the family \eqref{cycles} corresponding to $K=0$, that has been used decisively in the study of the analogous problem in dimension $N=1$. As we see here, the surface \eqref{surface} is quite far from being a cylinder (as we had in dimension $N=1$) but still has the same trace on the plane $\{X=0\}$. Moreover, let us also remark that this surface contains the points $P_0$ and $P_1$. The normal to this surface is given by
\begin{equation}\label{normal.surface}
\begin{split}
\overline{n}=&\left(-\frac{(\sigma+2)(2N+\sigma-2)}{4m}X-\frac{2N+\sigma-2}{2}H+\frac{2N+\sigma-2}{2}h_0,\right.\\
&\left.-2mH-\frac{2N+\sigma-2}{2}X+2mh_0,-1\right),
\end{split}
\end{equation}
and the flow of the system \eqref{PSsyst2} on the surface \eqref{surface} is given by the sign of the expression obtained, at the points of the surface, as the scalar product between $\overline{n}$ and the vector field of the system \eqref{PSsyst2}. After some rather tedious but direct calculations, we find that the flow is given by the sign of the following expression, that will be used throughout the paper:
\begin{equation}\label{flow.surface}
F(X)=-\frac{K_1(\sigma)}{2(m+1)}X-\frac{(2N-\sigma-6)(2N+\sigma-2)(\sigma+2)}{16m}X^3,
\end{equation}
where we recall that $K_1(\sigma)$ is defined in \eqref{coef}. Since $\sigma\in(0,\sigma_c)$, the coefficient of $X$ in \eqref{flow.surface} is positive, thus there exists a region very close to the plane $\{X=0\}$ where $F(X)\geq0$. Since the normal vector $\overline{n}$ points in negative direction in the third component, this gives that in such a region with $X$ sufficiently small, the surface \eqref{surface} cannot be crossed by connections "from below to above", that is, going from the region where $Z(X,H)<0$ in \eqref{surface} to the region where $Z(X,H)>0$.

\medskip

\noindent \textbf{Step 3. Connections to $P_1$ and from $P_0$.} We show that the orbits entering $P_1$ on its stable manifold and with $X>0$ very small, connect to $P_1$ from the region where $Z(X,H)>0$ with respect to the separatrix surface given by \eqref{surface}. To prove this claim, we calculate the difference between the approximation of the stable manifold entering $P_1$ obtained in \eqref{manif.P1} and the surface $Z(X,H)$. We have
\begin{equation}\label{dif.P1}
Z_1(X,H)-Z(X,H)=-\frac{K_1(\sigma)}{2}X\left[\frac{h_0}{3m+1}+\frac{H}{5m-1}+\frac{8mN-5m\sigma-18m+\sigma+2}{m(3m+1)(5m-1)}X\right],
\end{equation}
which is positive provided $X>0$ is sufficiently small, since for $\sigma\in(0,\sigma_c)$ we have $K_1(\sigma)<0$. We then infer from Step 2 above that an orbit entering $P_1$ on the stable manifold in a sufficiently small region $0<X<X_0$ such that $F(X)>0$ and $Z_1(X,H)-Z(X,H)>0$, cannot cross the surface $Z(X,H)$ in \eqref{surface} and has to stay above it (that is $Z_1(X,H)>Z(X,H)$) for any $H\in(0,2h_0)$, where we choose $2h_0$ since it is the $H$-component of the critical point $P_0=(0,2h_0,0)$ in the system \eqref{PSsyst2}. By performing totally similar calculations following again the theory in \cite[Section 2.7, p. 79-82]{Sh1}, we can also approximate the two-dimensional unstable manifold going out of $P_0$. To this end, we introduce the new variable $\overline{H}=Y-h_0$ and obtain a rather similar Taylor approximation up to order two of this unstable manifold, namely
\begin{equation}\label{manif.P0}
Z_0(X,\overline{H})=-AX-B\overline{H}+CX^2+D\overline{H}^2+EX\overline{H},
\end{equation}
where $A$, $B$, $C$, $D$, $E$ are the coefficients introduced in \eqref{interm19}. By changing $Z(X,H)$ into variables $(X,\overline{H})$, we can again calculate the difference between the Taylor approximation of the unstable manifold going out of $P_0$ and the separatrix surface \eqref{surface} to find
\begin{equation}\label{dif.P0}
Z_0(X,\overline{H})-Z(X,\overline{H})=-\frac{K_1(\sigma)}{2}X\left[-\frac{h_0}{3m+1}+\frac{\overline{H}}{5m-1}+\frac{8mN-5m\sigma-18m+\sigma+2}{m(3m+1)(5m-1)}X\right],
\end{equation}
which is negative provided $X>0$ is sufficiently small. It thus follows that there exists a fixed right-neighborhood of the plane $\{X=0\}$, in the form of a strip
$$
S:=\{0<X<\overline{X}(m,N,\sigma), -h_0\leq Y\leq h_0\},
$$
where $\overline{X}(m,N,\sigma)$ only depends on $m$, $N$, $\sigma$, such that the orbits entering $P_1$, respectively going out of $P_0$ on their two-dimensional manifolds inside the strip $S$ are completely separated by the surface \eqref{surface} and cannot connect.

\medskip

\noindent \textbf{Step 4. Connections between $Q_2$ and $P_1$.} We conclude from the previous analysis that the orbits entering $P_1$ from inside the strip $S$ have to cross the plane $\{Y=h_0\}$ at points with $Z>0$. Since the direction of the flow of the system \eqref{PSsyst} on the plane $\{Y=h_0\}$ is given by the sign of the expression
$$
-Z-(N-1)h_0X<0,
$$
it follows that these orbits come from the region $\{Y>h_0\}$. We also notice that, on orbits of the system \eqref{PSsyst} inside the region $\{Y>h_0\}$ we have
$$
\dot{Z}=Z[(m-1)Y+\sigma X]\geq0, \qquad \dot{Y}=1-\frac{m+1}{2}Y^2-Z-(N-1)XY<0,
$$
thus components $X$ and $Y$ are monotonic on these orbits. Such orbits cannot thus come from an $\alpha$-limit cycle and they have to come from a critical point inside the region $\{Y>h_0\}$, and the only such point is $Q_2$.

\medskip

\noindent \textbf{Step 5. The three sets argument.} Gathering the previous results, we have shown that in the two-dimensional manifold entering $P_1$ there are orbits coming from the unstable node $Q_5$ (with the convention of labeling made at the end of Lemma \ref{lem.Q4Q5}) and other orbits coming from the unstable node $Q_2$. At a formal level, since both $Q_5$ and $Q_2$ are nodes, the orbits coming from them are relatively open sets in the relative topology on the two-dimensional stable manifold of $P_1$, thus there must be a non-empty relatively closed set with orbits entering $P_1$ and coming from other critical points or $\alpha$-limits. This formal plan is made rigorous by a three-sets argument. Starting from the system \eqref{PSsyst2}, we keep only the first order approximations in the first and third equation and obtain in a neighborhood of $P_1$ that
$$
\frac{dZ}{dX}\sim\frac{(m-1)h_0Z}{(m-1)h_0X/2}=\frac{2Z}{X},
$$
whence we find by integration that the orbits entering $P_1$ on its two-dimensional stable manifold enter tangent to the beam of curves $Z=DX^2+o(X^2)$ with $D\in(0,\infty)$. Doing the same with the second equation and keeping only the first order terms, we further get
$$
\frac{dH}{dX}\sim-\frac{2(m+1)}{m-1}\frac{H}{X}-\frac{2(N-1)}{m-1},
$$
hence (taking into account that $H$, $X\to0$) we get the particular solution
$$
H=-\frac{2(N-1)}{3m+1}X+o(X).
$$
This latter equation, together with the one-parameter family $Z=DX^2+o(X^2)$, characterize the two-dimensional manifold entering $P_1$. Introduce then the following three sets:
\begin{equation*}
\begin{split}
&A=\{D\in(0,\infty): {\rm the \ orbit \ tangent \ to} \ Z=DX^2 \ {\rm comes \ from} \ Q_2\}, \\
&B=\{D\in(0,\infty): {\rm the \ orbit \ tangent \ to} \ Z=DX^2 \ {\rm does \ not \ come \ from} \ Q_5 \ {\rm or} \ Q_2 \}, \\
&C=\{D\in(0,\infty): {\rm the \ orbit \ tangent \ to} \ Z=DX^2 \ {\rm comes \ from} \ Q_5\}.
\end{split}
\end{equation*}
On the one hand, we infer from Step 4 that the set $A$ is non-empty and in fact contains an interval of the form $(D^*,\infty)$, since the orbits inside the strip $S$ will have $X>0$ as small as possible for a fixed $Z$, thus correspond to very large parameters $D$. On the other hand, we deduce from Step 1 that the set $C$ is non-empty and contains an interval of the form $(0,D_*)$, since the orbits lying in a tubular neighborhood of the unique orbit entering $P_1$ inside the plane $\{Z=0\}$ have very low component $Z$ for $X$ fixed and thus correspond to very small parameters $D>0$. Moreover, sets $A$ and $C$ are open, since $Q_2$ and $Q_5$ are unstable nodes. It then follows that the set $B$ is non-empty and closed, hence there are some orbits entering $P_1$ and coming from other critical points.

\medskip

\noindent \textbf{Step 6. End of the proof.} The connections corresponding to parameters $D\in B$ can be in one of the following four cases, as it follows from the local analysis in Sections \ref{sec.local} and \ref{sec.infty}:

$\bullet$ Orbits starting from $P_0$: they contain good profiles satisfying assumption (P3) in Definition \ref{def1}

$\bullet$ Orbits starting from $P_2$: they contain good profiles satisfying assumption (P2) in Definition \ref{def1}

$\bullet$ Orbits starting from $Q_1$: they contain good profiles satisfying assumption (P1) in Definition \ref{def1}

$\bullet$ Orbits coming from an $\alpha$-limit. We will rule out this possibility by using \cite[Theorem 1, Section 3.2]{Pe}, which gives that any $\alpha$-limit should be a compact set in the phase space, thus bounded both in $X$ and $Z$ components. We thus obtain that there exist $K_1$, $K_2>0$ such that
$$
\xi^{\sigma}f(\xi)^{m-1}\leq K_1, \qquad f(\xi)^{(m-1)/2}(\xi)\leq K_2\xi,
$$
which in particular gives that both $\xi$ and $f(\xi)$ remain bounded along the profile contained in such trajectory. Moreover, such a profile starting from a compact $\alpha$-limit present infinitely many oscillations, thus there exist an infinity of maxima and minima to $f(\xi)$, from which we can extract convergent sequences of minima, respectively maxima
$$
\xi_{n,\min}\to\xi_{\min}, \qquad \xi_{n,\max}\to\xi_{\max}\in(0,\infty),
$$
such that their terms are alternated (that is, a maximum point lies between two minimum points and viceversa). We then deduce that there exist points
$$
\xi_n\in(\xi_{n,\max},\xi_{n,\min}), \qquad f''(\xi_n)=0, \ \xi_n\to\xi_0,
$$
where the latter is obtained by eventually restricting ourselves to a subsequence, relabeled in the same way for simplicity. Replacing in the equation \eqref{ODE} we get
$$
\lim\limits_{n\to\infty}\frac{N-1}{\xi_n}(f^m)'(\xi_n)=\frac{1}{m-1}f(\xi_0)-\xi_0^{\sigma}f^m(\xi_0)\in\real.
$$
But it is easy to see that, if
$$
\xi_{\min}=\lim\limits_{n\to\infty}\xi_{n,\min}\neq\lim\limits_{n\to\infty}\xi_{n,\max}=\xi_{\max},
$$
then $(f^m)'(\xi_n)\to\pm\infty$ (the sign depending on whether the function oscillates from minima to maxima or viceversa when passing through $\xi_n$), which is a contradiction unless $\xi_0=+\infty$, which is not the case. It thus follows that the $\alpha$-limit should be reduced to a point to which the oscillations converge (being damped), thus a critical point. Since from all the remaining critical points $Q_1$, $P_0$ and $P_2$ start orbits containing good profiles, the proof is complete.
\end{proof}

\noindent \textbf{Remark.} An inspection of the previous proof shows that the dimension is irrelevant, thus the same proof gives the existence of good profiles with interface also for dimensions $N=2$ and $N=3$ when $\sigma\in(0,\sigma_c)$. This will be used later.

\section{Classification of profiles for $N\geq4$}\label{sec.class}

This is the longest and most technical chapter of this paper, which is devoted to the classification of the orbits entering $P_1$, with respect to their points of origin (among $Q_1$, $P_0$ and $P_2$, as it comes out from the proof in Section \ref{sec.exist}) for dimension $N\geq4$ and $\sigma\in(0,\sigma_c)$. The goal of this section is to prove Theorem \ref{th.class.N4}. Before starting the proof, we state a short but useful
\begin{lemma}\label{lem.cycles}
All the possible $\omega$-limit sets of the system \eqref{PSsyst} that are not a single point are included in the plane $\{X=0\}$.
\end{lemma}
\begin{proof}
We derive from Step 6 in the proof of Theorem \ref{th.exist.N4} in Section \ref{sec.exist} and from \cite[Theorem 2, Section 3.2]{Pe} that all the $\omega$-limit sets are included in either the invariant plane $\{X=0\}$ or the invariant plane $\{Z=0\}$ (otherwise, since the $\omega$-limit is itself an invariant set, going backwards on it, its $\alpha$-limit should be a compact set inside the phase space and we have shown that this is not possible). But it cannot belong to the plane $\{Z=0\}$ since a limit cycle must cross infinitely many times the plane $\{Y=0\}$, and this can be done only at points with $Z>1$.
\end{proof}
We are now in a position to classify the orbits entering $P_1$ in the phase space for $\sigma\in(0,\sigma_c)$ and prove the existence part in Theorem \ref{th.class.N4}. Due to its length, the proof will be again divided into several steps
\begin{proof}[Proof of Theorem \ref{th.class.N4}, existence part.]
\noindent \textbf{Step 1. Orbits from $P_2$ for $\sigma$ close to zero.} We recall the surface \eqref{surface}, that is fundamental for our analysis. In the initial variables $(X,Y,Z)$ it writes
\begin{equation}\label{surface2}
Z(X,Y)=-\frac{(\sigma+2)(2N+\sigma-2)}{8m}X^2-\frac{2N+\sigma-2}{2}XY-mY^2+\frac{2m}{m+1}.
\end{equation}
The aim of this first step of the proof is to prove that for $\sigma>0$ sufficiently small, the orbit going out of $P_2$ will connect to the critical point $P_3$ which behaves like an attractor. To this end, we begin by recalling that the direction of the flow of the system \eqref{PSsyst} on the surface \eqref{surface2} is given by the sign of the expression $F(X)$ in \eqref{flow.surface}, which is positive provided that
\begin{equation}\label{interm20}
X^2<X_0^2(m,N,\sigma)=-\frac{8mK_1(\sigma)}{(2N-\sigma-6)(2N+\sigma-2)(\sigma+2)(m+1)}>0.
\end{equation}
Notice here that dimension $N\geq4$ already plays an important role: the fact that $2N-\sigma-6>0$ which holds true for $0<\sigma<2(N-3)$ allows for the previous estimate. We next observe that $X(P_2)<X_0(m,N,\sigma)$ for $\sigma>0$ small, where $X_0(m,N,\sigma)$ is defined in \eqref{interm20}. Indeed, for $\sigma=0$ we have
$$
X_0(m,N,0)^2-X(P_2)^2=\frac{(3m-1)(m-1)(mN-N+m+3)}{2(mN-N+2)(m+1)(N-3)}>0,
$$
and this holds true also for $\sigma>0$ sufficiently small by continuity. Notice again the importance of the fact that $N-3>0$ in the previous estimate. Moreover, by replacing $X=X(P_2)$ and $Y=Y(P_2)$ in the expression of the surface \eqref{surface2} we get
\begin{equation}\label{surface.P2}
\begin{split}
Z(X(P_2),Y(P_2))&=\frac{m-1}{4m(mN-N+2)}\left[\frac{(3m-1)(m-1)(N-1)}{m+1}\right.\\&\left.-\frac{mN+4m-N}{2}\sigma
-\frac{m-1}{4}\sigma^2\right]>0,
\end{split}
\end{equation}
provided $\sigma>0$ is sufficiently small. Since $Z(P_2)=0$, it follows that the critical point $P_2$ lies below the surface \eqref{surface2} and inside the region where the flow on the surface satisfies $F(X)\geq0$.

\medskip

\noindent \textbf{Step 2. The unique orbit from $P_2$ goes to $P_3$ for $\sigma>0$ small.} This step is very technical and some calculations have been performed with the aid of a symbolic calculation software. Consider first the plane $\{Y=2X/(m-1)\}$ passing through the point $P_2$. The intersection of this plane with the surface \eqref{surface2} is given by the parabola
\begin{equation}\label{interm21}
Z=\frac{2m}{m+1}-\frac{R(\sigma)}{8m(m-1)^2}X^2,
\end{equation}
where
$$
R(\sigma)=(m-1)^2\sigma^2+2(m-1)(mN-4m-N)\sigma+4(m-1)(5m-1)N+4(3m^2+6m-1)>0,
$$
thus it is a parabola with a negative dominant coefficient and it intersects the plane $\{Z=0\}$ at
$$
X=X_1(m,N,\sigma), \qquad X_1(m,N,\sigma)^2=\frac{16m^2(m-1)^2}{(m+1)Q(\sigma)}.
$$
We next notice that $X_1(m,N,\sigma)<X_0(m,N,\sigma)$ for $\sigma>0$ sufficiently small. Indeed, we calculate
\begin{equation*}
\begin{split}
Z_0(\sigma)&=\frac{2m}{m+1}-\frac{Q(\sigma)}{8m(m-1)^2}X_0(m,N,\sigma)\\
&=\frac{(2mN+(m-1)\sigma+6m-2N+2)L(\sigma)}{(m-1)^2(m+1)(2N-6-\sigma)(2N+\sigma-2)(\sigma+2)},
\end{split}
\end{equation*}
with
$$
L(\sigma)=(m^2-1)\sigma^2+2(m+1)(mN+4m-N)\sigma-4(m-1)(3m-1)(N-1)<0
$$
provided $\sigma>0$ is sufficiently small, which implies also that $Z_0(\sigma)<0$ for $\sigma$ sufficiently small. Thus $X_1(m,N,\sigma)<X_0(m,N,\sigma)$ since the parabola \eqref{interm21} is decreasing with respect to $X$. We finally consider the plane $\{X=X_1(m,N,\sigma)\}$. It is obvious from geometrical reasons that $X_1(m,N,\sigma)>X(P_2)$ for $\sigma>0$ sufficiently close to zero, and the flow on this plane is given by the sign of the expression $(m-1)Y/2-X_1(m,N,\sigma)$ which is negative in the region where $Y<2X/(m-1)$. Moreover, the intersection of the plane $\{X=X_1(m,N,\sigma)\}$ with the surface \eqref{surface2} is a parabola having a positive maximum, namely
\begin{equation*}
Z=\frac{16m^2((m-1)\sigma+2mN+2m-2N+2)}{(m+1)L(\sigma)}-\frac{2N+\sigma-2}{2}X_1(m,N,\sigma)Y-mY^2.
\end{equation*}
We have thus built a compact region limited by the planes $\{X=0\}$, $\{Z=0\}$, $\{X=X_1(m,N,\sigma)\}$ and the surface \eqref{surface2}, having the critical point $P_2$ inside for $\sigma>0$ sufficiently small and from where no orbit can escape through its walls. Thus, since the orbits entering $P_1$ do that outside the surface \eqref{surface2}, the orbit going out of $P_2$ cannot connect to either $P_1$ or $Q_3$ for $\sigma>0$ sufficiently small, hence the analysis performed in Lemma \ref{lem.P3} and Proposition \ref{prop.cycles} shows that the orbit going out of $P_2$ must enter the critical point $P_3$.

\medskip

\noindent \textbf{Step 3. Orbits for $\sigma=\sigma_c$.} The goal of this step is to show that for $\sigma=\sigma_c$ the unique orbit going out of $P_2$ connects to the stable node $Q_3$. To this end, we notice first that for $\sigma=\sigma_c$ the critical point $P_2$ lies "above" the surface \eqref{surface2}. Indeed, if we replace $X=X(P_2)$, $Y=Y(P_2)$ and $\sigma=\sigma_c$ in \eqref{surface2} we get
$$
Z(X(P_2),Y(P_2))=-\frac{(N-1)(m-1)^2[m(N-4)+N-2]}{(mN-N+2)(3m+1)^2(m+1)}<0,
$$
since we are with $N\geq4$. We next prove that the orbits entering $P_1$ on the two-dimensional stable manifold do that "below" the surface \eqref{surface2}. To this end, for $\sigma=\sigma_c$ we approximate the stable manifold of $P_1$ by its Taylor development up to order three, a fact justified again by the theory in \cite[Section 2.7, p. 79-82]{Sh1}, to find that
$$
Z_1(X,Y)=AX+BY+CX^2+DH^2+EXH+FX^3+o(|(X,Y)^3|), \qquad H=Y+h_0,
$$
where
\begin{equation}\label{coef.P1}
\begin{split}
&A=\frac{4m(N-1)h_0}{3m+1}, \ B=2mh_0, \ C=\frac{2(mN+2m-N+2)(N-1)}{(3m+1)^2}, \ D=-m,\\
&E=-\frac{4(N-1)m}{3m+1}, \ F=-\frac{4(N-1)(mN+2m-N+2)[m(N-4)+N-2]\sqrt{2(m+1)}}{(5m-1)(3m+1)^2},
\end{split}
\end{equation}
and we just notice that the first five coefficients are exactly equal to the same ones of the surface \eqref{surface} (which is the same as \eqref{surface2} written in variables $(X,H,Z)$). Thus, we have
\begin{equation}\label{dif.P1bis}
Z_1(X,Y)-Z(X,Y)=FX^3+o(|(X,Y)^3|)<0, \qquad {\rm since} \ F<0 \ {\rm for} \ N\geq4.
\end{equation}
We deduce that the stable manifold of $P_1$ lies below the surface $Z(X,Y)$ in \eqref{surface2} in a neighborhood of $P_1$. Finally, recalling that the flow of the system \eqref{PSsyst} over the surface \eqref{surface2} is given by \eqref{flow.surface}, which becomes for $\sigma=\sigma_c$
$$
F(X)=-\frac{4(N-1)(mN+2m-N+2)[m(N-4)+N-2]}{(3m+1)^3}X^3<0,
$$
and that the direction of the normal to the surface $\overline{n}$ in \eqref{normal.surface} points in the negative direction of the $Z$ component, we conclude that no trajectory in the phase space can cross this surface from above (that is, from the region $\{Z>Z(X,Y)\}$) to below (that is, into the region $\{Z<Z(X,Y)\}$). This proves that the orbit going out of $P_2$ will never cross this surface and thus it cannot connect either to $P_1$ or to $P_3$. We then infer from the local analysis and from Lemma \ref{lem.cycles} that it cannot reach a limit cycle (since all the limit cycles in \eqref{cycles} with $K>0$ lie below the surface), thus it has to enter the unstable node $Q_3$.

\medskip

\noindent \textbf{Step 4. A new three-sets argument.} Since $Q_3$ is a node, it follows by standard continuity that there exists $\sigma^0<\sigma_c$ such that the orbit going out of $P_2$ enters $Q_3$ for any $\sigma\in(\sigma^0,\sigma_c)$. Moreover, we have proved that there exists $\sigma_0>0$ such that the orbit going out of $P_2$ enters $P_3$ for any $\sigma\in(0,\sigma_0)$. We can thus define the three sets
\begin{equation*}
\begin{split}
&A=\{\sigma\in(0,\sigma_c): {\rm the \ orbit \ going \ out \ of} \ P_2 \ {\rm enters} \ P_3\}, \\
&B=\{\sigma\in(0,\sigma_c): {\rm the \ orbit \ going \ out \ of} \ P_2 \ {\rm enters} \ P_1\}, \\
&C=\{\sigma\in(0,\sigma_c): {\rm the \ orbit \ going \ out \ of} \ P_2 \ {\rm enters} \ Q_3\}.
\end{split}
\end{equation*}
and conclude that $A$ and $C$ are open sets (the former follows from the attractor-like behavior of $P_3$ established in Lemma \ref{lem.P3} and the latter since $Q_3$ is an attractor) and $(0,\sigma_0)\subseteq A$, $(\sigma^0,\sigma_c)\subseteq C$. We infer that the set $B$ is non-empty and closed and it contains at least a point, for which there is a good profile with interface contained in the orbit going out of $P_2$ and satisfying assumption (P2) in Definition \ref{def1}. Moreover, since $A$ and $C$ are open, at least their limit points belong to $B$, thus $\sigma_0\in B$ and $\sigma^0\in B$, completing the proof of the third statement in Theorem \ref{th.class.N4}.

\medskip

\noindent \textbf{Step 5. Connections to $P_1$ from $P_0$ and $Q_1$.} For the range $\sigma\in(0,\sigma_0)$ for which the construction done in Steps 1 and 2 applies, it is easy to see that the same argument limits all the orbits going out of $P_0$, as they have to enter necessarily the same compact region which is used to limit the orbits going out of $P_2$, thus all the orbits from $P_0$ connect to $P_3$. We infer by elimination that the orbits with good profiles entering $P_1$ according to Theorem \ref{th.exist.N4} must come from $Q_1$ and thus contain profiles that satisfy assumption (P1) in Definition \ref{def1}.

Passing to $\sigma<\sigma_c$ but close to $\sigma_c$, we claim that \emph{no orbit coming out of $Q_1$ can connect to $P_1$}. To this end, we restrict ourselves to the case $\sigma=\sigma_c$ and we show that the whole two-dimensional surface coming out of $Q_1$ lies at a strictly positive distance from $P_1$. By passing to the limit as $X\to\infty$ and $Y\to0$ in the equation of the surface \eqref{surface2} we obviously get that $Q_1$ lies above the surface \eqref{surface2}, thus all the orbits going out of $Q_1$ connect to $Q_3$ for $\sigma=\sigma_c$, and we work on the (compact, if seen in the Poincar\'e hypersphere) two-dimensional invariant manifold $M$ formed by all the orbits going out of $Q_1$ with the flow given by the dynamical system on it, which also gives an orientation to the manifold. This manifold is limited by two separatrices, one being included in the Poincar\'e hypersphere connecting $Q_1$ to $Q_5$ and then $Q_3$ and the other being formed by the orbit connecting $Q_1$ to $P_2$ inside the plane $\{Z=0\}$ (constructed in Lemma \ref{lem.P0Q1Z0}) continued with the orbit going out of $P_2$ and entering $Q_3$. If the critical point $P_1$ is a limit of some sequence of points on this manifold, it means that in any neighborhood of $P_1$ a saddle sector is formed by the orbits in $M$, whence $P_1\in M$ and a new separatrix connecting $Q_1$ to $P_1$ must exist inside the manifold $M$. But such separatrix does not exist for $\sigma=\sigma_c$ and we reach a contradiction. It thus follows that the orbits going out of $Q_1$ stay at a positive distance from $P_1$ for $\sigma=\sigma_c$, and by continuity and a standard argument of tubular neighborhoods it readily follows that the same holds true in a small interval $\sigma\in(\sigma_2,\sigma_c)$ for some $\sigma_2\in(0,\sigma_c)$. Taking into account that there exists a orbit entering $P_1$ and containing good profiles, by elimination this orbit can come only from $P_0$, as claimed, for any $\sigma>0$ sufficiently close to $\sigma_c$.
\end{proof}
We show in Figure \ref{fig2} a picture of the phase space in both cases discussed in the classification done before: when $\sigma>0$ is small and when $\sigma<\sigma_c$ is very close to the critical exponent (for the data of the experiments, $N=4$ and $m=2$, we have $\sigma_c=6/7$). Notice the number of oscillations the orbits going out of $P_0$ do, before entering $P_3$. 

\begin{figure}[ht!]
  \begin{center}
  \subfigure[$\sigma>0$ small]{\includegraphics[width=7.5cm,height=6cm]{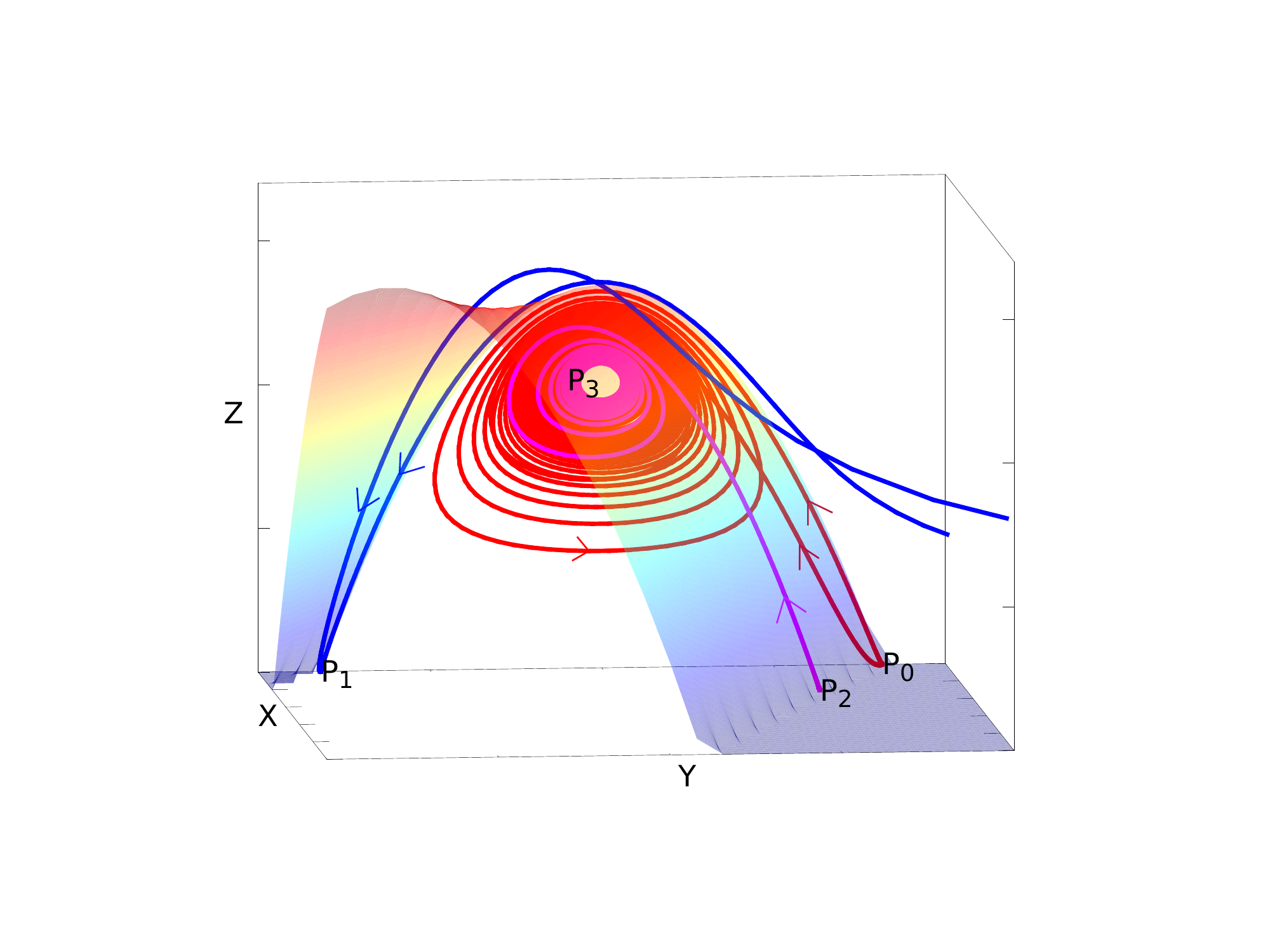}}
  \subfigure[$\sigma<\sigma_c$, $\sigma\sim\sigma_c$]{\includegraphics[width=7.5cm,height=6cm]{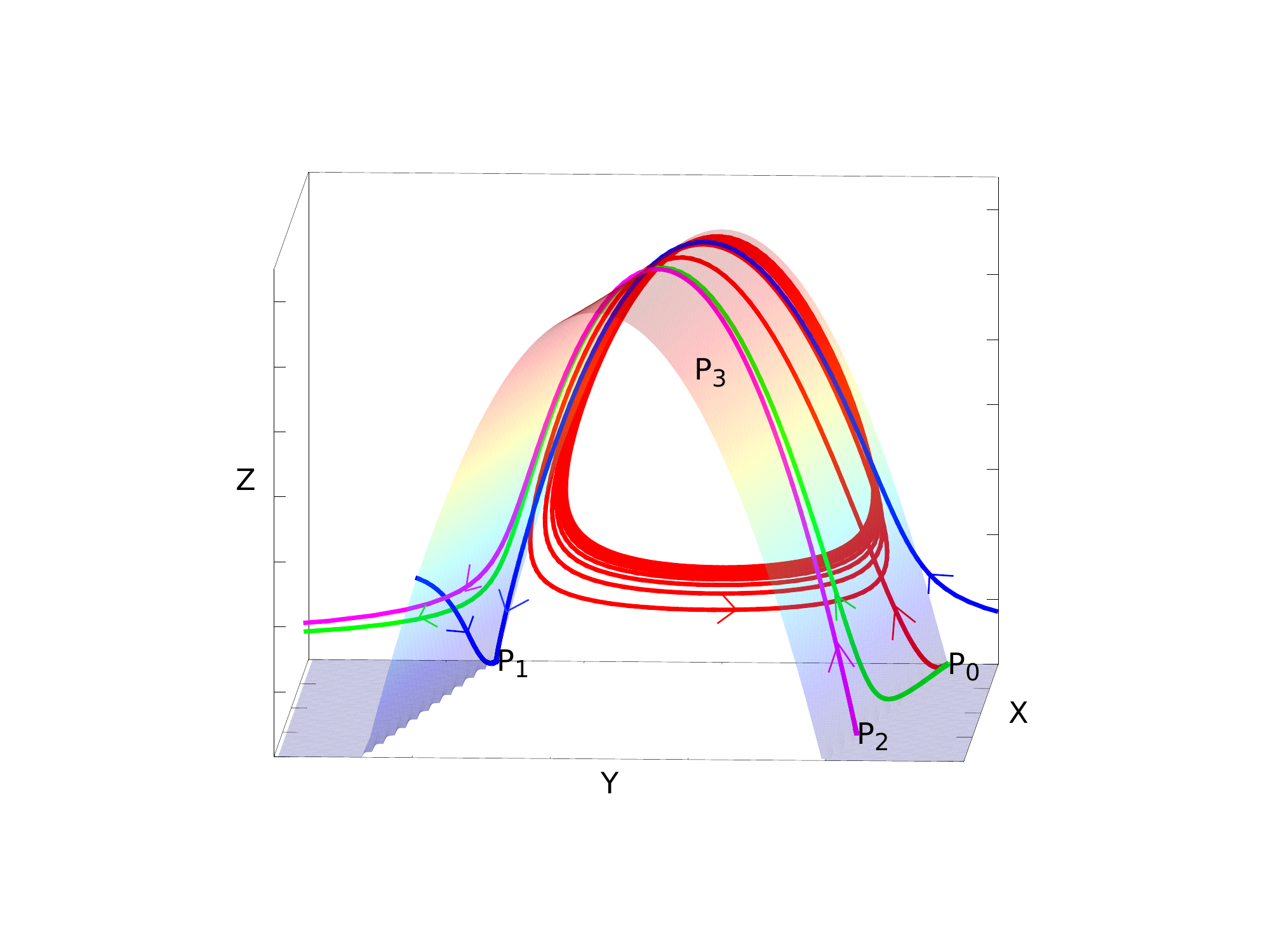}}
  \end{center}
  \caption{Trajectories in the phase space for different values of $\sigma\in(0,\sigma_c)$. Numerical experiment for $m=2$, $N=4$ and $\sigma=0.5$, respectively $\sigma=0.84$}\label{fig2}
\end{figure}

\noindent \textbf{Remarks.} 1. In fact, one can notice that the manifold $M$ associated with its natural flow inherited from the system \eqref{PSsyst} is a \emph{parallel manifold} of strip form, according to the theory in \cite{Ne75}, and for such manifolds Markus-Neumann's Theorem states in particular that they are topologically equivalent to regions of the plane having the same separatrix configuration \cite{Ne76} (see also \cite[Section 3.11]{Pe}). Since the two \emph{limit-separatrices} of $M$ both connect to $Q_3$ for any $\sigma\in(0,\sigma_c)$ for which the orbit going out of $P_2$ connects to $Q_3$, this fact strongly suggests that \textbf{all} the orbits going out of $Q_1$ will enter $Q_3$ whenever the unique orbit going out of $P_2$ enters $Q_3$, and thus the outcome of Step 5 can be extended to the whole interval $\sigma\in(\sigma^0,\sigma_c)$ for which $P_2$ connects to $Q_3$. However, this is not a full proof of this fact.

\medskip

2. The fact that $m(N-4)+N-2>0$ appeared in a decisive way in some steps of the proof of Theorem \ref{th.class.N4}. We see again here that, if we allow fractional dimensions, the branching point for the classification is given by $N^*=(4m+2)/(m+1)$, a fact noticed in the Introduction in \eqref{critN} with the mapping \eqref{transf}.

We complete the proof of Theorem \ref{th.exist.N4} with its non-existence part, that we state here as a separate result for simplicity.
\begin{proposition}\label{prop.nonexist}
For $N\geq4$ and $\sigma\in[\sigma_c,2(N-3)]$, there are no good profile with interface solutions to Eq. \eqref{ODE}.
\end{proposition}
\begin{proof}
This is now rather easy in view of the ideas and calculations used in the proof of Theorem \ref{th.class.N4}. The plan is to prove that for $\sigma\geq\sigma_c$ and $N\geq4$ the orbits coming from $P_0$, $P_2$ or $Q_1$ cannot enter $P_1$; in fact, they will all connect to $Q_3$, while as an outcome of the proof, all the two-dimensional manifold entering $P_1$ will come from the unstable node $Q_5$ (always with the convention of identification between $Q_5$ and $Q_4$ introduced after Lemma \ref{lem.Q4Q5}). First of all, for $\sigma>\sigma_c$, recalling the local approximation $Z_1(X,H)\equiv Z_1(X,Y)$ of the manifold entering $P_1$ given by \eqref{manif.P1}, we readily get from \eqref{dif.P1} and the fact that $K_1(\sigma)>0$ that
$$
Z_1(X,Y)-Z(X,Y)<0,
$$
where $Z(X,Y)$ is the expression of the surface \eqref{surface2}, thus for any $\sigma\geq\sigma_c$ we infer that the orbits entering $P_1$ do that from "below" the surface \eqref{surface2} (the exact equality case $\sigma=\sigma_c$ being considered in \eqref{dif.P1bis}). With respect to the point $P_0$, for $\sigma=\sigma_c$ we can again approximate the two-dimensional manifold going out of it up to order three as we did for the manifold entering $P_1$ and get
$$
Z_0(X,Y)=-AX-BY+CX^2+DH^2+EXH-FX^3+o(|(X,Y)^3|), \qquad H=Y-h_0,
$$
with coefficients $A$ to $F$ given in \eqref{coef.P1}, noticing thus that
\begin{equation}\label{dif.P0bis}
Z_0(X,Y)-Z(X,Y)=-FX^3+o(|(X,Y)^3|)>0, \qquad {\rm since} \ F<0 \ {\rm for} \ N\geq4.
\end{equation}
We conclude that the orbit going out of $P_0$ does this above the surface \eqref{surface2} for $\sigma=\sigma_c$ and it is easy to check that the same occurs for any $\sigma>\sigma_c$ using only the approximation up to order two in \eqref{manif.P0} and the sign in \eqref{dif.P0}. With respect to the point $P_2$, we obtain from the calculation in \eqref{surface.P2} that for $\sigma=\sigma_c$ we have
$$
Z(X(P_2),Y(P_2))=-\frac{(m-1)^3(N-1)[m(N-4)+N-2]}{(mN-N+2)(m+1)(3m+1)^2}<0,
$$
thus the critical point $P_2$ having $Z(P_2)=0$ lies above the surface \eqref{surface2} for any $\sigma>\sigma_c$, since the expression \eqref{surface.P2} is decreasing with respect to $\sigma$. It remains to show that the orbits going out of $Q_1$ also lie above this surface. The profiles going out of $Q_1$ satisfy, according to Lemma \ref{lem.Q1}, $f(0)=A>0$ and $f'(0)=0$, thus in terms of the phase space variables have $X=+\infty$, $Y=Z=0$. Letting thus points of the form $(x_0,0,0)$ with $x_0$ very large and replacing in the formula of \eqref{surface2} we get
$$
Z(x_0,0)=\frac{2m}{m+1}-\frac{(\sigma+2)(2N+\sigma-2)}{8m}x_0^2<0,
$$
thus these points approximating $Q_1$ lie above the surface \eqref{surface2} and passing to the limit $x_0\to+\infty$, the same occurs for $Q_1$. It is then easy to notice that for $\sigma\in[\sigma_c,2(N-3)]$ the sign of \eqref{flow.surface} is always negative, thus the surface \eqref{surface2} cannot be crossed by any trajectory of the system from above to below (or rigorously from the region $\{Z>Z(X,Y)\}$ into the region $\{Z<Z(X,Y)\}$). Hence, no orbit from either $P_0$, $P_2$ or $Q_1$ can enter $P_1$ for $\sigma\in[\sigma_c,2(N-3)]$, as claimed.
\end{proof}

\section{Global analysis in dimensions $N=2$ and $N=3$}\label{sec.N23}

Let us take now $N=2$ or $N=3$. Our goal is to prove Theorem \ref{th.exist.N23}. This is an immediate consequence of the following
\begin{proposition}\label{prop.existN23}
For any $\sigma\in(0,\sigma_c]$, there exist good orbits entering $P_1$. All these orbits come from the critical point $Q_1$.
\end{proposition}
\begin{proof}
Recalling the surface \eqref{surface2} and the flow of the system \eqref{PSsyst} on it, given by the function $F(X)$ defined in \eqref{flow.surface}, the fundamental remark is that the flow $F(X)$ \textbf{is positive} for any $\sigma\in(0,\sigma_c]$. This is due to the fact that for $N\in\{2,3\}$ we have $2N-\sigma-6<0$ for any $\sigma>0$. We thus deduce that the trajectories of the system \eqref{PSsyst} cannot cross the surface \eqref{surface2} from the interior (that is, from the region $\{Z<Z(X,Y)\}$ towards the exterior (that is, the region $\{Z<Z(X,Y)\}$) of it. Let us also remark at this point that in dimensions $N\leq3$, the surface \eqref{surface2} can be also written in the form
\begin{equation}\label{parab.ell}
Z=\frac{2m}{m+1}-m\left[Y+\frac{2N+\sigma-2}{4m}X\right]^2-\frac{(2N+\sigma-2)(\sigma+6-2N)}{16m}X^2,
\end{equation}
becoming an elliptic paraboloid and we can properly speak about interior and exterior of it. A surface having similar geometric form is represented in Figure \ref{fig3} below. It is then easy to check that the proof of Theorem \ref{th.exist.N4} (done in Section \ref{sec.exist}) simply applies also to dimensions $N=2$ and $N=3$ and moreover, it can be extended also to the limit case $\sigma=\sigma_c$. Indeed, for $\sigma<\sigma_c$ the approximations to the two-dimensional manifolds of $P_0$ and $P_1$ given in \eqref{manif.P0} and \eqref{manif.P1}, and the fact that the orbits entering $P_1$ do that from the exterior of the surface \eqref{surface2} as it follows from \eqref{dif.P1}, respectively the ones going out of $P_0$ do that in the interior of the surface \eqref{surface2} as it follows from \eqref{dif.P0}, hold true. But as a significant difference with respect to dimensions $N\geq4$, the same holds true also for $\sigma=\sigma_c$. This is due to the fact that the coefficient
$$
F=-\frac{4(N-1)(mN+2m-N+2)[m(N-4)+N-2]\sqrt{2(m+1)}}{(5m-1)(3m+1)^2}
$$
appearing in \eqref{dif.P1bis} changes sign and becomes \emph{positive} for $N=2$ and $N=3$ (and in fact for any fractional dimension $N<(4m+2)/(m+1)$ if we consider $N$ as a parameter in Eq. \eqref{ODE}). This implies that also for $\sigma=\sigma_c$ the orbits entering $P_1$ remain in the exterior of the surface \eqref{surface2}, while the ones going out of $P_0$ remain in the interior (as it follows from \eqref{dif.P0bis}), and the proof of the existence of a good orbit entering $P_1$, done in Section \ref{sec.exist}, can be directly extended to $\sigma=\sigma_c$. We moreover deduce that there can be no connection $P_0-P_1$ for any $\sigma\in(0,\sigma_c]$, since the orbits are separated by the surface \eqref{surface2}. The same happens for the orbits going out of $P_2$. Indeed, we recall from \eqref{surface.P2} that $Z(X(P_2),Y(P_2))$ is strictly decreasing with $\sigma>0$ and that for $\sigma=\sigma_c$ we get
$$
Z(X(P_2),Y(P_2))=-\frac{(N-1)(m-1)^2[m(N-4)+N-2]}{(mN-N+2)(3m+1)^2(m+1)}>0,
$$
since now $m(N-4)+N-2<0$. This important change implies that $P_2$ is in the interior region of the surface \eqref{surface2} and thus the orbit going out of it cannot cross the surface in order to enter $P_1$. By elimination, it follows that all the good orbits entering $P_1$, for any $\sigma\in(0,\sigma_c]$, come from $Q_1$.
\end{proof}
As a remark, the proof of existence of a good orbit entering $P_1$ can be extended above $\sigma=\sigma_c$. This follows from the fact that, for $\sigma=\sigma_c$ there exist still orbits entering $P_1$ and coming from the unstable nodes $Q_2$ and $Q_5$, as it comes out by an inspection of the proof of Theorem \ref{th.exist.N4} (see Section \ref{sec.exist}) together with the adaptations for $\sigma=\sigma_c$ explained above. Thus, we can extend the existence of such connections by standard continuity up to some $\sigma_2>\sigma_c$, and the three sets argument proving that there exists at least a good orbit (Step 5 in Section \ref{sec.exist}) is still available for any $\sigma\in(\sigma_c,\sigma_2)$. 

\section{Non-existence of good blow-up profiles for $\sigma$ large}\label{sec.nonexist}

This final section is devoted to the proof of Theorem \ref{th.nonexist}. To show that no good profiles with interface exist, we have to rule out, for $\sigma$ sufficiently large, connections $P_2-P_1$, $Q_1-P_1$ and $P_0-P_1$. The approach will be similar for all these points, since we will follow, for $\sigma>0$ sufficiently large, the orbits starting from their crossing point with the plane $\{Y=0\}$ in the phase space, a plane that all of them must cross at least once. As usual, we divide the proof into several steps for simplicity.

\medskip

\begin{proof}[Proof of Theorem \ref{th.nonexist}]
\noindent \textbf{Step 1. A change of variable.} In order to simplify the presentation, since we are dealing with $\sigma$ very large, we let $U=\sigma X$ and we transform the system \eqref{PSsyst} into
\begin{equation}\label{PSsyst4}
\left\{\begin{array}{ll}\dot{U}=\frac{m-1}{2}UY-\lambda U^2,\\
\dot{Y}=-\frac{m+1}{2}Y^2+1-Z-(N-1)\lambda UY,\\
\dot{Z}=Z[(m-1)Y+U],\end{array}\right.
\end{equation}
with $\lambda=1/\sigma$. We thus have to study trajectories of the system \eqref{PSsyst4} for $\lambda$ sufficiently small. Notice that the flow of the system \eqref{PSsyst4} over the surface \eqref{surface2} becomes
\begin{equation}\label{flow.surface2}
F(U)=\frac{K(\lambda)}{2(m+1)}U-\frac{(2\lambda+1)(2\lambda N-2\lambda+1)(2\lambda N-6\lambda-1)}{16m}U^3,
\end{equation}
as it readily follows from \eqref{flow.surface}, where
$$
K(\lambda)=2\lambda(m-1)(N-1)-(3m+1).
$$
We remark here that for $\sigma>2(N-3)$ if $N\geq4$ or for $\sigma>0$ if $N<4$, the surface \eqref{surface2} is an elliptic paraboloid, a fact that can be seen immediately from \eqref{parab.ell} since the only thing that counts is the sign of $\sigma+6-2N$ which is now positive. We further observe that $F(U)<0$ in \eqref{flow.surface2} if $U\in(0,U_0)$, where
\begin{equation}\label{interm24}
U_0^2=\frac{8mK(\lambda)}{(m+1)(2\lambda+1)(2\lambda N-2\lambda+1)(2\lambda N-6\lambda-1)}>0.
\end{equation}
The intersection of the surface \eqref{surface2} with the plane $\{U=U_0\}$ is a parabola
\begin{equation*}
\begin{split}
Z(U_0,Y)&=\frac{2m}{m+1}-mY^2-\sqrt{\frac{2m(2\lambda N-2\lambda+1)K(\lambda)}{(m+1)(2\lambda+1)(2\lambda N-6\lambda-1)}}Y\\
&-\frac{K(\lambda)}{(m+1)(2\lambda N-6\lambda-1)},
\end{split}
\end{equation*}
having a maximum equal to
$$
M(\lambda)=\frac{2\lambda mN+6\lambda m-2\lambda N+2\lambda+m-1}{2(2\lambda+1)(m+1)},
$$
and we infer that
$$
M(\lambda)-1=\frac{2\lambda(m-1)(N+1)-(m+3)}{2(2\lambda+1)(m+1)}<0, \qquad {\rm for} \ \lambda\in\left(0,\frac{m+3}{2(m-1)(N+1)}\right).
$$
It follows that the intersection between the surface \eqref{surface2} and the plane $\{U=U_0\}$ lies strictly below the plane $\{Z=1\}$ for $\lambda>0$ sufficiently small as above and, since the surface is an elliptic paraboloid, the same holds true for the surface \eqref{surface2} for any $U\geq U_0$.

\medskip

\noindent \textbf{Step 2. An orbit cannot cross the surface \eqref{surface2} in the region $\{Y>0\}$.} To prove this statement, we find that the intersection of the surface \eqref{surface2} with the plane $\{Y=0\}$ is given by
$$
Z(U,0)=\frac{2m}{m+1}-\frac{(2\lambda+1)(2\lambda N-2\lambda+1)}{8m}U^2>0
$$
provided that
$$
0<U^2<U_1^2=\frac{16m^2}{(m+1)(2\lambda+1)(2\lambda N-2\lambda+1)}.
$$
Recalling that $U_0$ is given in \eqref{interm24}, we notice that
$$
U_1^2-U_0^2=\frac{8m[2\lambda(mN-5m+N-1)+m+1]}{(m+1)(2\lambda+1)(2\lambda N-2\lambda+1)(2\lambda N-6\lambda-1)}<0
$$
for $\lambda>0$ sufficiently small, due to the sign of $2\lambda N-6\lambda-1<0$. Since the amplitude with respect to the component $U$ of the surface \eqref{surface2} for some $\{Y>0\}$ is smaller than the one attained in the intersection with the plane $\{Y=0\}$ and the latter is smaller than $U_0$ as shown above, it follows that the part of the surface included in the half-space $\{Y>0\}$ is also included in the strip $\{0\leq U\leq U_0\}$ where the direction of the flow does not allow for the surface to be crossed from the exterior towards its interior.

\medskip

\noindent \textbf{Step 3. Crossing the plane $\{Y=0\}$.} We already have seen in the proof of Proposition \ref{prop.nonexist} that the orbits going out of $P_0$, $P_2$ and $Q_1$ are above the surface \eqref{surface2} near their starting points. It then follows from Step 2 that they cannot cross the surface before reaching the plane $\{Y=0\}$. The direction of the flow on this plane is given by the sign of $1-Z$, thus the orbits have to cross the plane $\{Y=0\}$ at points with height $Z>1$. If the crossing point of such orbit with the plane $\{Y=0\}$ finds itself in the strip $\{0\leq U\leq U_0\}$, it will forever remain in this strip after crossing since the component $U$ is decreasing in the half-space $\{Y<0\}$, as we infer from the first equation in \eqref{PSsyst4}. Thus, such an orbit cannot cross the surface \eqref{surface2} in order to go to $P_1$. If an orbit intersects the plane $\{Y=0\}$ at a point with $U>U_0$ and $Z>1$, then this orbit has to cross first the plane $\{Z=1\}$ in order to cross afterwards the surface \eqref{surface2}, since we have shown in Step 1 that the region $\{U>U_0\}$ of the surface \eqref{surface2} lies below the plane $\{Z=1\}$, and this intersection with the plane $\{Z=1\}$ should still lie in the region $\{U>U_0\}$, as otherwise we are already in the strip in which the surface cannot be crossed. But in order to intersect the plane $\{Z=1\}$, the component $Z$ of the trajectory must start to decrease. This can only happen in the region where $\dot{Z}<0$, that is, $(m-1)Y+U<0$, which writes
\begin{equation}\label{interm25}
Y<-\frac{U}{m-1}<-\frac{U_0}{m-1}.
\end{equation}

\medskip

\noindent \textbf{Step 4. Point of no return.} We consider now the plane $\{Y=-U_0/(m-1)\}$ that the orbits have to cross in order to allow component $Z$ to decrease, according to \eqref{interm25}. The flow of the system \eqref{PSsyst4} on this plane is given by the sign of
$$
H(Z,U)=1-\frac{m+1}{2(m-1)^2}U_0^2-Z+\frac{(N-1)\lambda}{m-1}U_0U.
$$
Noticing that
$$
1-\frac{m+1}{2(m-1)^2}U_0^2=B(\lambda,m,N)=\frac{A(\lambda,m,N)}{(m-1)^2(2\lambda+1)(2\lambda N-2\lambda+1)(2\lambda N-6\lambda-1)},
$$
where
\begin{equation*}
\begin{split}
A(\lambda,m,N)&=8(N-1)(N-3)(m-1)^2\lambda^3+4(N^2-4N+1)(m-1)^2\lambda^2\\
&-2(m-1)(4mN-N-3)\lambda+11m^2+6m-1,
\end{split}
\end{equation*}
we get that $H(Z,U)<0$ for $\lambda$ sufficiently small since
$$
\lim\limits_{\lambda\to0}B(\lambda,m,N)=-\frac{11m^2+6m-1}{(m-1)^2}<0.
$$
We end the proof by also noticing that $-U_0/(m-1)<-h_0$ for $\lambda>0$ sufficiently small. This is obtained from the fact that
$$
\frac{U_0^2}{(m-1)^2}-h_0^2=-\frac{2}{m+1}B(\lambda,m,N)>0
$$
for $\lambda>0$ small. It then follows that once crossed, the plane $\{Y=-U_0/(m-1)\}$ cannot be crossed again by a trajectory from the left to the right, and that the point $P_1$ lies in the region $\{Y>-U_0/(m-1)\}$, thus no trajectory crossing the plane $\{Y=0\}$ can enter this point.
\end{proof}
\begin{figure}[ht!]
  \begin{center}
  \includegraphics[width=10cm,height=7.5cm]{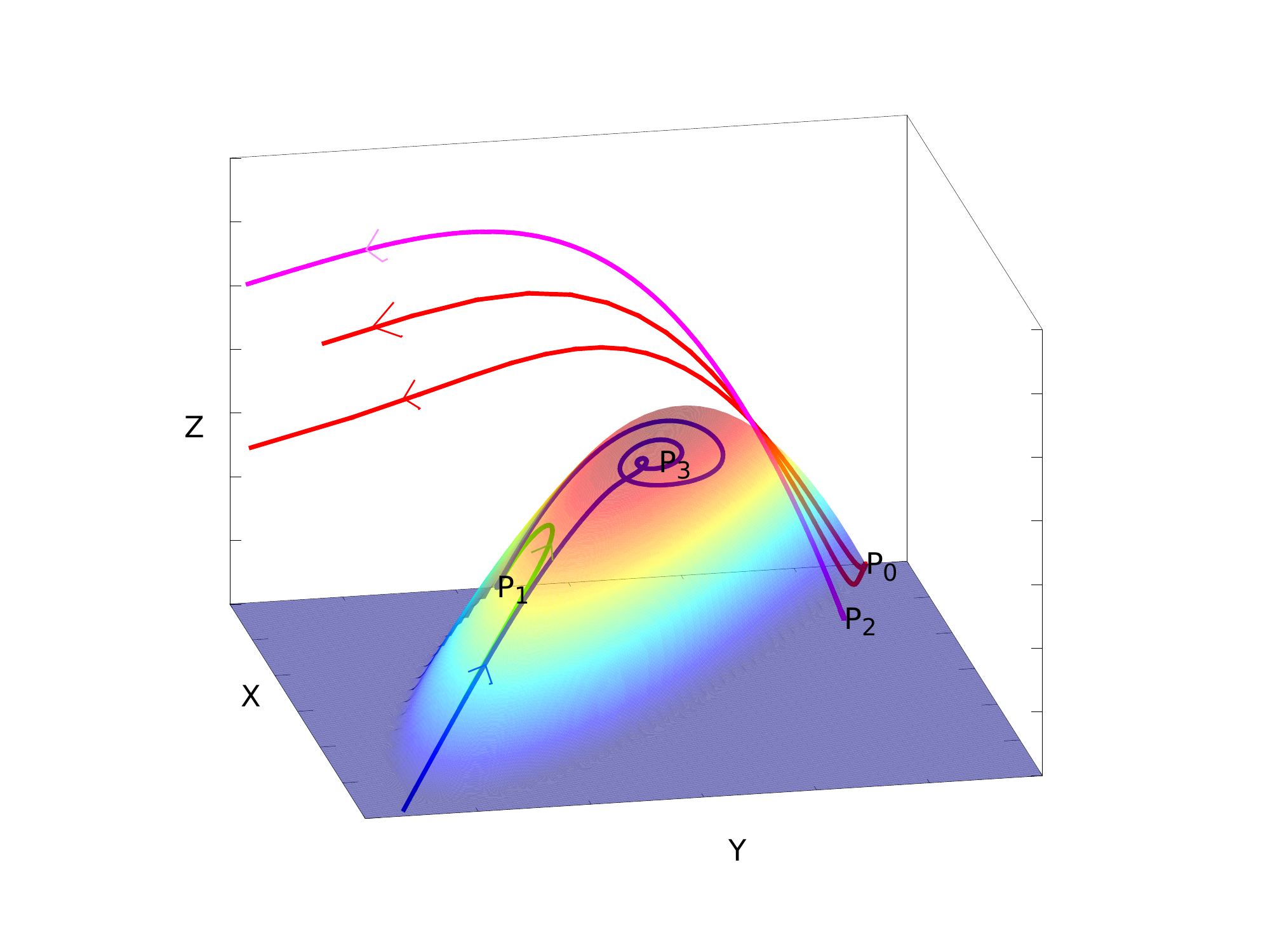}
  \end{center}
  \caption{The phase space and the separatrix surface, in form of an elliptic paraboloid, for $\sigma$ sufficiently large. Numerical experiment for $m=2$, $N=4$ and $\sigma=5$}\label{fig3}
\end{figure} 
We end this final section by plotting in Figure \ref{fig3} the phase space for $\sigma>\sigma_c$ sufficiently large. We see some orbits entering $P_1$ and coming from $Q_5$, while the orbits going out of $P_0$ and $P_2$ stay above the elliptic paraboloid and go away from $P_1$.

\bigskip

\noindent \textbf{Acknowledgements} A. S. is partially supported by the Spanish project MTM2017-87596-P.

\end{document}